\documentclass[12pt,leqno]{amsart}
\usepackage{amssymb,amsmath,amsthm,bm,scalerel,stackengine}
\usepackage{amssymb,amsmath,amsthm,bm,scalerel,stackengine}
\usepackage{comment}
\usepackage{graphicx}
\usepackage{color}
\usepackage[textsize=tiny]{todonotes}
\usepackage[normalem]{ulem}
\usepackage[bookmarksopen,bookmarksdepth=3,colorlinks,citecolor=red,pagebackref,hypertexnames=ue]{hyperref}
\usepackage[msc-links,nobysame,non-sorted-cites, initials]{amsrefs}
\usepackage[inline]{enumitem}
\usepackage{mathtools}
\mathtoolsset{showonlyrefs}
\usetikzlibrary{quotes,arrows.meta}

\definecolor{darkblue}{RGB}{0,0,160}

\usepackage[lining]{libertine}   
\usepackage{cabin}
\usepackage[libertine]{newtxmath}

\usepackage[T1]{fontenc}

\def\eps{\varepsilon}
\def\d{{\rm d}}
\def\R {\mathbb{R}}
\def\N {\mathbb{N}}

\def\Tr {{\mathrm{Tr}}}
\def\supp {{\mathrm{supp}\,}}
\def\Z {{\mathbb Z}}

\def\CC{\mathbb{C}}

\newcommand{\TT}{\mathbb{ T} }

\newcommand{\F}{\mathcal{F}}

\def\ind{\cic{1}}
\newcommand{\cic}{\bm}

\numberwithin{equation}{section}
\numberwithin{counter2}{section}
\newtheorem{proposition}[subsection]{Proposition}
\newtheorem{theorem}[counter]{Theorem}

\newtheorem{lemma}[subsection]{Lemma}

\theoremstyle{definition}
\newtheorem{definition}[subsection]{Definition}
\newtheorem*{remark*}{Remark}
\newtheorem*{warn*}{A word of warning}

\newtheorem{remark}[subsection]{Remark} 
\theoremstyle{plain}

\newcommand{\norm}[1]{\left\Vert #1\right\Vert}
\newcommand{\ep}{\varepsilon}

\newcommand{\floor}[1]{\left\lfloor #1 \right\rfloor}
\newcommand{\ceiling}[1]{\left\lceil #1 \right\rceil}
\usepackage{scalerel} 

\newcommand{\EE}{\mathbb{E}}

\newcommand{\de}{\delta}

\newcommand{\beqq}{\begin{align*}}
	\newcommand{\eeqq}{\end{align*}}
    \DeclareMathOperator{\Log}{Log}
\title{ Polynomial Ergodic Averages Along  Short Intervals 
}

\author[A. Fragkos]{Anastasios Fragkos}
\address[A. Fragkos]{Department of Mathematics,
Georgia Tech \\
686 Cherry Street, Atlanta, GA 30332-0160 USA}
\author[H. Mousavi]{Hamed Mousavi}
\address[H. Mousavi]{Department of Mathematics,
University of Bristol \\
 Woodland Road, Bristol, BS8 1UG}
\author[A. Stokolosa]{Amelia Stokolosa}
\address[A. Stokolosa]{Department of Mathematics,
Georgia Tech \\
686 Cherry Street, Atlanta, GA 30332-0160 USA}

\begin{document}
    	\begin{abstract}
We study pointwise convergence of polynomial ergodic averages over short intervals whose left endpoints tend to infinity. For a polynomial orbit of degree $d\geq2$ and doubly lacunary starting times, we prove $L^p$ variational estimates, and hence almost-everywhere convergence, for $1<p<\infty$ in the range $c>(d-1)/d$. This gives the first pointwise ergodic theorem for polynomial orbits along short intervals. We also show that the endpoint $L^1$ fails along every infinite subsequence. In a different direction, we prove that substantially denser sequences of starting times exhibit the strong sweeping-out property.
	\end{abstract}
	\maketitle

%\tableofcontents

	\section{Introduction}

Birkhoff's ergodic theorem is the starting point for pointwise convergence of ergodic averages \cite{Birkhoff1931}. For nonlinear arithmetic times, Bourgain's
foundational work established almost-everywhere convergence of polynomial
ergodic averages for $L^p$, $p>1$, and initiated the circle-method approach to
such questions \cite{Bourgain_88_Lp,Bourgain_88_max,Bourgain_89_arithmetic_sets}.
Quantitative and variational refinements have since become an important part of
the theory; see, for instance,
\cite{Krause2014,MirekSteinTrojan2017,MirekSteinZorinKranich2020}. Buczolich and Mauldin proved that the square
averages are universally $L^1$-bad, in the sense that on every non-atomic
ergodic system there exists an $L^1$ function for which the corresponding
averages fail to converge almost everywhere. LaVictoire extended this mechanism
to $d$-th powers and the primes, showing moreover that the badness persists
along every subsequence \cite{BM,L9}. 

Let $(X,\mathcal B,\mu,T)$ be an invertible measure-preserving system, let
$P\in\mathbb Z[t]$ have degree $d\geq2$, and fix $0<c<1$.  We write
\[A_N^{P,c}f(x):=\frac1{\floor{N^c}}\sum_{h=1}^{\floor{N^c}}
f\bigl(T^{P(N+h)}x\bigr).
\]
The harmless choice of one endpoint of the short interval will be used
interchangeably below.
For $\lambda \in \N $, set
\[
N_k=N_k(\lambda):=\lambda^{\lambda^k},
\qquad
\mathbb D_\lambda:=\{N_k:k\in\mathbb N\}.
\]
Thus $\mathbb D_\lambda$ is doubly lacunary: its logarithms are themselves
lacunary.

For a single function,
the short-interval polynomial analogue of von Neumann's theorem follows from
the spectral theorem and Weyl's criterion: after spectral reduction, one is
precisely estimating polynomial exponential sums on intervals
$[N,N+\floor{N^c}]$.  For multiple functions, the Host--Kra theory and Leibman's
extension to arbitrary F{\o}lner sequences give the corresponding $L^2$ norm
convergence \cite{HostKra2005,LeibmanFolner2005}; indeed,
$[N,N+\floor{N^c}]$ is a F{\o}lner sequence in $\mathbb Z$ for every $c>0$.  In the
arithmetically thinner prime setting, the short-interval uniformity estimates
of Matom\"aki, Shao, Tao and Ter\"av\"ainen imply $L^2$ convergence for bounded
multiple ergodic averages (and along primes) in windows with
$N^{5/8+\varepsilon}\leq H\leq N^{1-\varepsilon}$
\cite{MatomakiShaoTaoTeravainen2023}.  These results also indicate why
quantitative short-interval questions are governed by exponential-sum
cancellation.

The sweeping-out phenomenon sharply separates norm from pointwise convergence: moving intervals may preserve von Neumann's $L^2$ convergence while destroying Birkhoff-type almost-everywhere convergence. For the moving averages
\[
A_{N}^{L_N} f(x) \coloneqq \frac{1}{L_N}\sum_{h \in (0,L_N]} f(T^{N+h} x),
\]
the work of Bellow--Jones--Rosenblatt and of Rosenblatt--Wierdl shows that pointwise convergence depends sensitively on the relative position and length of the averaging intervals $(N,N+L_N]$, with suitable geometric separation yielding convergence and its failure allowing sweeping out \cite{BJR_90_moving_avgs,Rosenblatt_Wierdl_ergodic_95}.
 Polynomial orbits add
an arithmetic, multifrequency layer that is invisible to this single frequency
setting.  This leads to the question motivating the paper: for which short
windows and starting times can a Bourgain-type pointwise polynomial ergodic
theorem still hold?

Our first result gives such a theorem on the doubly lacunary set
$\mathbb D_\lambda$ ,  introduced above, in the stronger form of a variational estimate.

\begin{theorem}\label{t:short-variation}
	Let $P\in\mathbb Z[t]$ have degree $d\geq2$, and let $\lambda \in \N$,
	$1<p<\infty$, $\rho>2$, and $c\in((d-1)/d,1)$.  Then
	\[
	\left\|\mathcal V^\rho
	\bigl(A_N^{P,c}f:N\in\mathbb D_\lambda\bigr)
	\right\|_{L^p(X)}
	\lesssim_{p,\rho,P,c,\lambda}\|f\|_{L^p(X)}.
	\]
	In particular, $A_N^{P,c}f$ converges almost everywhere as
	$N\to\infty$ through $\mathbb D_\lambda$.
\end{theorem}

The assumption $\rho>2$ is sharp in the variational parameter: a uniform estimate
 cannot hold for $\rho<2$; one may compare the sharp endpoint obstruction even
for the long averages with polynomial orbits in \cite{Krause2014}.  The condition $p>1$ is also sharp, in a stronger sense.  Even arbitrary further
sparsification of the starting times does not allow an almost everywhere pointwise convergence theorem at $L^1$ for
monomial orbits.

\begin{theorem}\label{t:short-l1-bad}
	Let $d\geq2$, $P(n)=n^d$ , $0<c<1$, and let
	$\mathcal N\subseteq\mathbb N$ be any infinite set.  The subsequence $\{P(n)\}$ 
    %family
	%$\{A_N^{P,c}:N\in\mathcal N\}$ 
    is universally $L^1$-bad: on every 
	aperiodic ergodic system there is an $f\in L^1(X)$ for which
	$A_N^{P,c}f$ fails to converge almost everywhere  as 
	$N\to\infty$ through $\mathcal N$.
\end{theorem}

In particular, no $L^1$ maximal or variational estimate corresponding to
Theorem~\ref{t:short-variation} can hold.

A natural next question is how sparse the sampling times must be in order to retain almost-everywhere convergence on $L^p, p>1$. In light of the criterion of Rosenblatt and Wierdl for ordinary moving averages, it is natural to ask whether an analogous condition is sufficient in the present polynomial setting \cite{Rosenblatt_Wierdl_1992}.

\begin{theorem}\label{t:dense-sweeping}
	Let $(j_n)$ be a strictly increasing sequence of positive integers such that
	$j_n\to\infty$ and
	\[
	\frac{\log j_{n+1}}{\log j_n}\longrightarrow1.
	\]
	For every $P\in\mathbb Z[t]$ with $\deg P\geq2$ and every $0<c<1$,
	the family $\{A_{j_n}^{P,c}\}$ has the strong sweeping-out
	property for every aperiodic transformation of a non-atomic standard
	probability space.  Thus, for every
	$\delta>0$, there is a set $E$ with $0<\mu(E)<\delta$ such that
	\[
	\limsup_{n\to\infty}A_{j_n}^{P,c}\mathbf 1_E=1,
	\qquad
	\liminf_{n\to\infty}A_{j_n}^{P,c}\mathbf 1_E=0
	\]
	almost everywhere.
\end{theorem}

After discarding finitely many terms, the hypothesis includes
$j_n=2^{q(n)}$ for every nonconstant
$q\in\mathbb Z[t]$ with positive leading coefficient, as well as
$j_n=n!$.  Consequently these denser time sets fail even for bounded
functions. The factorial example highlights a genuine discrepancy with the
Rosenblatt--Wierdl moving-interval theory: the sequence \(n!\) satisfies their convergence
condition, but Theorem~\ref{t:dense-sweeping} shows that polynomial orbits along
these starting times still sweep out. The reason is the multifrequency obstruction created by the
polynomial phase; see Proposition~\ref{p:sweepingout}.

Several natural questions remain.  First, we are cautiously optimistic that the restriction
$c>(d-1)/d$ in Theorem~\ref{t:short-variation} is sharp.  Second, one may ask for bilinear and,
more generally, multilinear analogues of the short-interval pointwise and
variational theorem, in the spirit of the recent full-interval polynomial
theory \cite{KMT_pw_erg_bilinear_22,KMTT_bil_poly_prime_25,kosz2025multilinearcirclemethodquestion}.  Third, it would be
interesting to establish a polynomial Wiener--Wintner theorem along these short
windows, with an exceptional set independent of the polynomial phase \cite{Lesigne1993,Frantzikinakis2006}.

We conclude the introduction with an outline of the arguments.  For
Theorem~\ref{t:short-variation}, transference reduces matters to
$\ell^p(\mathbb Z)$, where we perform a short-interval variant of the
Hardy--Littlewood circle method. We then decompose according to
the complexity of the canonical fractions and apply
\cite[Theorem~5.2]{KMSZ26}, the Ionescu--Wainger theorem for canonical
fractions and Gauss sum decay makes the complexity sum convergent.  The condition
$c>(d-1)/d$ enters precisely in the short exponential sum estimates concerning the minor arc components.

Next, Sawyer's weak-type principle together with Conze transference reduces
Theorem~\ref{t:short-l1-bad} to
the failure of a uniform weak-$(1,1)$ estimate on finite cyclic groups
\cite{Sawyer1966,BM,L9}.  We recycle LaVictoire's family lemma, modifying only the range on which local periodicity is required so that it covers all polynomial values arising from the short interval. This allows the short-interval averages to be reduced to the same finite arithmetic model used in LaVictoire's construction, while the rest of the argument remains essentially unchanged. The new
arithmetic ingredient is a reduction connecting the cyclic model to the actual
short-interval maximal operator.  

Finally, for Theorem~\ref{t:dense-sweeping}, we apply the sweeping-out criterion \cite[Theorem~1.14]{SSOP_96}. The key point is to choose scales $j_{k_1}<\cdots<j_{k_N}$ and frequencies $\xi_1,\ldots,\xi_N$ such that the multiplier at $\xi_i$ is close to $1$ at the scale $j_{k_i}$, but close to a fixed nontrivial phase at all the other selected scales. The assumption
\[
\frac{\log j_{n+1}}{\log j_n}\to1
\]
allows us to choose the scales with enough separation while keeping their relative growth under control. At these scales,\[
m_{j_k}(\xi)
=e(a_dj_k^d\xi)
+O(j_k^{d-1+c}|\xi|),
\]
and we construct the frequencies recursively on $\mathbb T$ so that the leading phases have exactly the required diagonal/off-diagonal behavior. This verifies the criterion of \cite{SSOP_96} and yields strong sweeping out.

Section~2 fixes the notation, variational norms, canonical fractions and
Ionescu--Wainger projections.  Section~3 proves the variational estimate by the
circle method.  Section~4 establishes the universal $L^1$ counterexample in
the finite cyclic model and transfers it to ergodic systems.  Section~5 proves
strong sweeping out for a class of denser sequences of sampling times.
	\label{s:methodsandstructure
}

\subsection{Acknowledgment}  
We would like to thank Máté Wierdl for constructive conversations and for
bringing the problem to our attention.

\section{Notation}

We collect the notation and conventions used throughout the paper.

\subsection{General Conventions}

For a positive integer $H$, let
\[
[H]:=\{1,\ldots,H\}.
\]
Whenever a sum or an expectation is taken over integers in an interval $I\subseteq\R$, we write $n\in I$ in place of $n\in I\cap\Z$. For every finite nonempty set $I\subseteq\Z$, we set
\[
\mathop{\EE}\limits_{n\in I} f(n)
:=\frac{1}{|I|}\sum_{n\in I}f(n),
\qquad
\EE_H f:=\mathop{\EE}\limits_{h\in[H]}f(h).
\]

If $A$ and $B$ are nonnegative quantities, then $A\lesssim B$ means that $A\leq CB$ for some constant $C>0$, and $A\gtrsim B$ means that $B\lesssim A$. We write $A\lesssim_{\vartheta}B$ when the constant is allowed to depend on a parameter $\vartheta$, and write $A\simeq B$ when both $A\lesssim B$ and $B\lesssim A$ hold. We use $A\sim B$ synonymously with $A\simeq B$. Unless indicated otherwise, implicit constants may depend on parameters fixed in the surrounding argument, but not on the variables being estimated.

We set
\[
e(x):=e^{2\pi i x},
\qquad
\langle x\rangle \coloneqq \max \left\{1,|x| \right\},
\qquad
\Tr_h f(x):=f(x-h).
\]
We identify the torus $\TT:=\R/\Z$ with $[-1/2,1/2)$ and use normalized Lebesgue measure on $\TT$. Differences of elements of $\TT$ are always represented in $[-1/2,1/2)$.

For $t>0$, we write
\[
{\rm Log\,}t:=\floor{\log_2 t}.
\]
Recall that, for $\lambda>1$,
\[
N_k=N_k(\lambda):=\lambda^{\lambda^k}
\qquad
\mathbb D_\lambda:=\{N_k:k\in\N\}.
\]

\subsection{Fourier transforms}

For a finitely supported function $f\colon\Z\to\CC$, define
\[
\F_{\Z}f(\xi)=\widehat f(\xi)
:=\sum_{n\in\Z}f(n)e(-n\xi),
\qquad \xi\in\TT.
\]
For $g\in L^1(\TT)$, define
\[
\F_{\Z}^{-1}g(x)=g^{\vee}(x)
:=\int_{\TT}g(\xi)e(x\xi)\,\d\xi,
\qquad x\in\Z.
\]
Thus $(\widehat f)^{\vee}=f$ for every finitely supported $f\colon\Z\to\CC$.

For $f\in L^1(\R)$, define
\[
\F_{\R}f(\xi)=\widehat f(\xi)
:=\int_{\R}f(x)e(-x\xi)\,\d x,
\qquad \xi\in\R,
\]
and, for $g\in L^1(\R)$, define
\[
\F_{\R}^{-1}g(x)=g^{\vee}(x)
:=\int_{\R}g(\xi)e(x\xi)\,\d\xi,
\qquad x\in\R.
\]

For $X=\Z$, let $\widehat X=\TT$, and for $X=\R$, let $\widehat X=\R$. Given a bounded symbol $m$ on $\widehat X$, we define the Fourier multiplier operator $T_X[m]$ by
\[
\F_X\bigl(T_X[m]f\bigr)(\xi)
:=m(\xi)\F_Xf(\xi),
\qquad \xi\in\widehat X,
\]
initially for finitely supported functions when $X=\Z$ and for Schwartz functions when $X=\R$.

\subsection{Variational seminorms}

Let $\mathbb I\subseteq\N$ be ordered by the usual order, and let $\{f_K:K\in\mathbb I\}$ be a collection of measurable functions on a measure space $X$. For $1\leq\rho<\infty$, define the pointwise $\rho$-variation seminorm by
\[
\mathcal V^{\rho}(f_K:K\in\mathbb I)(x)
:=\sup_{\substack{J\geq1\\K_0<\cdots<K_J\\K_j\in\mathbb I}}
\left(
\sum_{j=0}^{J-1}
\left|f_{K_{j+1}}(x)-f_{K_j}(x)\right|^{\rho}
\right)^{1/\rho}.
\]
For $\rho=\infty$, we use the usual modification
\[
\mathcal V^{\infty}(f_K:K\in\mathbb I)(x)
:=\sup_{\substack{J\geq1\\K_0<\cdots<K_J\\K_j\in\mathbb I}}
\max_{0\leq j<J}
\left|f_{K_{j+1}}(x)-f_{K_j}(x)\right|.
\]

\subsection{Ionescu-Wainger projections}

Fix a smooth even function $\eta\colon\R\to[0,1]$ satisfying
\[
\ind_{[-1/4,1/4]}\leq\eta\leq\ind_{(-1/2,1/2)},
\]
and, for $m\in\Z$, write
\[
\eta_{\leq m}(\xi):=\eta(2^{-m}\xi).
\]

Following the notation of \cite{KMSZ26}, for every integer $l\geq1$ define the set of canonical fractions by
\begin{equation}\label{eq canonical fracs}
\Sigma_{\leq l}
:=\left\{
\frac{a}{q}\in\mathbb Q/\Z:
1\leq q\leq 2^l,\ (a,q)=1
\right\}.
\end{equation}
We set $\Sigma_{\leq0}:=\varnothing$ and
\[
\Sigma_l:=\Sigma_{\leq l}\setminus\Sigma_{\leq l-1}.
\]
For integers $l\geq1$ and $m\in\Z$, define the dyadic major arcs by
\[
\mathcal M_{\leq l,\leq m}
:=\bigcup_{\theta\in\Sigma_{\leq l}}
\left\{\xi\in\TT:|\xi-\theta|\leq2^m\right\}.
\]
If $m\leq-2l-2$, then the arcs in this union are pairwise disjoint.

Let $\Sigma\subseteq\mathbb Q/\Z$ be finite, let $G\colon\Sigma\to\CC$, and let $\mathfrak m$ be a bounded function on $\R$. For $X\in\{\Z,\R\}$, we write
\[
T_X\llbracket G,\mathfrak m\rrbracket_{\Sigma}
:=T_X\left[\sum_{\theta\in\Sigma}
G(\theta)\mathfrak m(\,\cdot-\theta)\right].
\]
Equivalently,
\begin{align*}
\F_{\Z}\left(T_{\Z}\llbracket G,\mathfrak m\rrbracket_{\Sigma}f\right)(\xi)
&=\sum_{\theta\in\Sigma}G(\theta)\mathfrak m(\xi-\theta)\F_{\Z}f(\xi),\\
\F_{\R}\left(T_{\R}\llbracket G,\mathfrak m\rrbracket_{\Sigma}f\right)(\xi)
&=\sum_{\theta\in\Sigma}G(\theta)\mathfrak m(\xi-\theta)\F_{\R}f(\xi).
\end{align*}
Here and below, elements of $\mathbb Q/\Z$ are represented in $[-1/2,1/2)$ when they occur in expressions on $\R$.

The dyadic Ionescu--Wainger projections are defined by
\[
\F_{\Z}\left(\Pi_{\leq l,\leq m}f\right)(\xi)
:=\sum_{\theta\in\Sigma_{\leq l}}
\eta\left(2^{-m}(\xi-\theta)\right)\widehat f(\xi),
\]
and, at a fixed complexity scale, by
\[
\F_{\Z}\left(\Pi_{l,\leq m}f\right)(\xi)
:=\sum_{\theta\in\Sigma_l}
\eta\left(2^{-m}(\xi-\theta)\right)\widehat f(\xi).
\]
In particular,
\[
\supp\F_{\Z}\left(\Pi_{\leq l,\leq m}f\right)
\subseteq\mathcal M_{\leq l,\leq m}.
\]

For a function $G\colon\mathbb Q/\Z\to\CC$, write $G_q$ for its restriction to $\Z/q \Z $  and define
\[
U_G
:=\sup_{q\geq1}
\norm{\F^{-1}_{\Z/q\Z}G_q}_{L^1(\Z/q\Z)},
\]
where the supremum runs over integers $q\geq1$. If $U_G<\infty$ and $\Sigma\subseteq\mathbb Q/\Z$ is finite, set
\begin{equation}\label{eq CG}
C_G(\Sigma):=U_G\norm{G}_{\ell^{\infty}(\Sigma)}.
\end{equation}

In this paper, $G$ denotes the polynomial Gauss sum
\[
G\left(\frac{a}{q}\right)
:=\frac{1}{q}\sum_{r\in[q]}
e\left(-\frac{a}{q}P(r)\right),
\]
for $a\in\Z$ and integers $q\geq1$, where $P\in\Z[t]$ has degree $d\geq2$. This defines a function on $\mathbb Q/\Z$, independently of the chosen representation of the rational class.

	\section{Sharp Variational $L^p$ bounds}
    In this section we will give a short proof of the variational $L^p$ estimates of our main result. As our main result concerns times in the set $ \mathbb{D}_{\lambda}$ whenever we write $N$ we mean that $N \in \mathbb{D}_{\lambda}$ for a fixed $\lambda>1$ and therefore the dependence of the implied constants on the parameter $\lambda$ will not be tracked.
	
	The objective of this section is to prove the following proposition. \begin{proposition} \label{p:varintlp} Whenever $\rho>2$ we have that 
		\[ \left \| \mathcal{V}^\rho(A_N^{P,c}f: N \in \mathbb{D}_{\lambda}) \right \|_{\ell^p(\Z)} \lesssim \left \| f \right \|_{\ell^p(\Z)}. \]
	\end{proposition} As outlined in the introduction, we will start by carrying out the Hardy-Littlewood circle method for the averages \[A_{N}^{P,c}f=\EE_{n \in [N,N+N^c]}f(x-P(n))= \int_{\TT} \widehat{f}(\xi) \mathbf{m}_N^{P,c}(\xi)e(\xi x) \d \xi, \text{ where }\mathbf{m}_{N}^{P,c}(\xi) \coloneqq \EE_{n \in [N,N+N^c] }e(-\xi P(n)).   \] 
	Before proceeding with the proof, we fix the data
	\[
	\lambda>1, \quad \rho>2,\quad
	p\in(1,\infty),\quad
	P(x)=a_d x^d+a_{d-1}x^{d-1}+\cdots+a_0\in\mathbb{Z}[x],\quad
	c\in\left(\frac{d-1}{d},1\right),
	\]
	and  $r \in \N $ sufficiently large so that
	\[
	(2r)' \leq p \leq 2r.
	\]
	These quantities will remain fixed throughout the proof and, whenever there is no risk of ambiguity, will be suppressed from the notation. It will also be beneficial for us in order to set up the stage for the application of \cite[Theorem 5.2]{KMSZ26} to introduce the auxiliary parameter $0<\alpha<\min \left\{(10^6 dr)^{-1},  (100c^{-1}(d-1))^{-1},\frac{cd}{10(d-1)}-\frac{1}{10},\frac{cd-(d-1)}{10(d-1)+1} \right\}$ as well as the notation $	l_{(N)} \coloneqq \Log N^\alpha \text{ and }  L_{(N)} \coloneqq  \Log N+10l_{(N)}.$

	At scale $N$ the major arc for our averages is $\mathcal{M}_{\leq l_{(N)},-(d-1)L_{(N)}}.$ We require the following ingredients which will be expanded as lemmata below: minor arc control and the structure of the $\mathbf{m}_N$ on the major arcs. 
	
	The lemma below shows that away from the major arcs the size of $\mathbf{m}_N$ is small.  
	\begin{lemma}\label{l:minorarcalldegree}
		For $\xi \notin \mathcal{M}_{\leq l_{(N)}, \leq -(d-1)L_{(N)}}$, we have
		\begin{align*}
			\left| \mathbf{m}_N(\xi)\right| \lesssim N^{-\delta}, 
		\end{align*}
		for some $\delta >0$. 
	\end{lemma}
	\begin{proof}
		We perform the change of variables $n \mapsto N+u$ to rewrite the multiplier as a normalized exponential sum of length $\floor{N^c}$
		\begin{align*}
			\mathbf{m}_N(\xi)= \frac{1}{N^c} \sum_{u \in [0,N^c]} e \left(-\xi P(N+u) \right)+O(N^{-c}).
		\end{align*}
		Observe that $\xi P(N+u)$ is a polynomial of degree $d$ with leading coefficient $\xi a_d$. By Dirichlet's approximation theorem, one can find  $ q \in \left[2^{(d-1)L_{(N)}}\right]$ such that $\left| \xi a_d-\frac{a}{q} \right| \leq \frac{1}{q2^{(d-1)L_{(N)}}}$. As such, $ \left| \xi-\frac{a}{qa_d} \right| \leq \frac{1}{q2^{(d-1)L_{(N)}}} \leq \frac{1}{q^2}$. We can write $\frac{a}{qa_d}=\frac{a'}{q a_d'}$ with $(a',qa_d')=1$. Since $\xi \not \in \mathcal{M}_{\leq l_{(N)}, \leq -(d-1)L_{(N)}}$, it must be that $q \gtrsim N^{\alpha}$. In turn, we must have $q a_d' \gtrsim N^{\alpha}.$ Using Weyl's inequality, see for example  \cite[Theorem 4.3]{Nathanson1996}, it follows that 
		\begin{align*}
			|\mathbf{m}_N(\xi)| \lesssim   \frac{1}{N^c} N^{c(1+\frac{\alpha}{10002^{100d}})} \left( \frac{1}{N^{\alpha}} +\frac{1}{N^c}+\frac{N^{(d-1)(1+10\alpha)}}{N^{cd}} \right)^{\frac{1}{2^{d-1}}} \lesssim N^{-\delta}
		\end{align*}
		for some $\delta>0$. 
	\end{proof}
	The continuous counterpart of our averages on the frequency side is $V_N$ defined by the formula \begin{align*}
		V_N(\xi) \coloneqq  \frac{1}{N^c}  \int_{N}^{N+N^c}e(-\xi a_d t^d) \d t. 
	\end{align*} Standard considerations reveal its oscillatory and stationary behavior which we may quantify as follows \begin{equation}
		\label{eq:vn}
		V_{N}(\xi)= \begin{cases}
			O\left ( \left| N^{d-1+c} \xi  \right|^{-1} \right) \\ 
			1+O\left (|N^d\xi| \right ). 
		\end{cases}
	\end{equation}
	From the following proposition we learn the structure of the multiplier on a major arc.
	\begin{proposition}\label{p:majorarcalldeg}
		Let $\frac{A}{Q} \in \Sigma_l$, where $ N \geq 2^{l/\alpha} $ and $|\xi - A/Q| \leq 2^{-(d-1)L_{(N)}} $. We have 
		\begin{align}\label{eq approx on major arcs}
			\left|	\mathbf{m}_N(\xi) -G\left (\frac{A}{Q}\right ) V_N \left( \xi - \frac{A}{Q}\right) \right|  \lesssim N^{-\delta}. 
		\end{align}
		%	for some $\de >0$. 
	\end{proposition}
	\begin{proof}
		Let $ \xi=\frac{A}{Q}+h$. For $n \in [N]$ and $1 \leq r \leq Q$, we have 
		\begin{align*}
			& \frac{A}{Q} P(kQ+r) \equiv \frac{A}{Q}P(r) \mod 1, \\
			& hP(n+r) =h a_dn^d+O\left(h N^{d-1}Q \right).
		\end{align*}
		We rewrite the multiplier $\mathbf{m}_N$ as follows. 
		\begin{align*}
			& \frac{1}{N^c} \sum_{N \leq n \leq N+N^c}e\left(- \left(\frac{A}{Q}+h\right) P(n) \right)=\frac{1}{N^c} \sum_{0 \leq r<Q} \sum_{ \frac{N-r}{Q} \leq k \leq \frac{N+N^c-r}{Q} } e\left(-  \left(\frac{A}{Q}+h\right) P(kQ+r)  \right) \\ 
			& = \frac{1}{N^c} \sum_{0 \leq r<Q} \left(\sum_{ \frac{N-r}{Q} \leq k \leq \frac{N+N^c-r}{Q} }\left(e\left( -\frac{A}{Q} P(r)-ha_d(kQ)^d \right) +O\left(hN^{d-1}Q\right)\right) \right) \\ 
			& =\frac{1}{N^c} \sum_{0 \leq r<Q}e\left(-\frac{A}{Q} P(r)\right) \left(\sum_{\frac{N-r}{Q} \leq k \leq \frac{N+N^c-r}{Q}}\left(e(-ha_d(kQ)^d)\right) \right)+ O\left(hN^{d-1}Q\right) . 
		\end{align*}
		The right-hand side of the equality above is equal to 
		\begin{align*}
			%		& =\sum_{0 \leq r<Q}e\left(\frac{A}{Q}p(r)\right)\frac{1}{QN^c}\int_{N}^{N+N^c} e(ha_dt^d) \d t  + O\left(  2^l N^{-c} + |h| N^{d-1} 2^l  \right)\\ 
			&  G\left (\frac{A}{Q}\right ) V_{N}(h)  + O\left(  2^l N^{-c} + |h| N^{d-1} 2^l  \right) 
		\end{align*}
		by utilizing the relation \[ \sum_{n=a}^bf(n)=\int_{a}^bf(x) \d x +O( \left \| f \right \|_{\infty}+(b-a+1) \left \| f' \right \|_{\infty} ).  \]

	\end{proof}
	
	We conclude the series of auxiliary lemmata with one that establishes the $L^p(\R)$ boundedness of the continuous analogue  of the maximal operator in question, which follows by \cite{SteinRealVariableMethods}[Chapter II, Theorem 4] however we choose to present an alternative argument.
	\begin{lemma} \label{l:maxfunctionvariant} For every $p>1$, we have that 
		\[ \left \|  \sup_{N \in \mathbb{D}_{\lambda}} |f* V_N^{\vee}|  \right \|_{L^{p}(\R)} \lesssim \left \| f \right \|_{L^p(\R)}. \] 
	\end{lemma}
	\begin{proof}
		A change of variables allows us to focus on the case $a_d=1.$	It is clear that it suffices to establish the weak type $(1,1)$ bound. We fix a level $\sigma>0$ and consider the level set \[ E_{\sigma} \coloneqq \left\{x : \sup_{N \in \mathbb{D}_{\lambda}} \frac{1}{N^{d-1+c}}\int_{N^d}^{(N+N^c)^d}|f(x-y)| \d y > \sigma   \right\}. \]  If $x \in E_{\sigma}$ then there exists $N_x \in \mathbb{D}_{\lambda}$ with the property that  \begin{equation} \label{eq:bigaverage}
			\frac{1}{N_x^{d-1+c}} \int_{x-[N_x^d,(N_{x}+N_x^{c})^d]}|f|(y) \d y > \sigma
		\end{equation} which implies that $I_x \coloneqq x-[N_x^d,(N_x+N_x^c)^{d}] \subseteq \left\{\mathrm{M}f > c' \sigma \right\}$, for a very small positive constant $c'.$ Now, if we take a Whitney decomposition $\mathcal{W}$ of $\left\{\mathrm{M}f > c'  \sigma\right\}=\bigcup_{Q \in \mathcal{W}}Q$  and we have that $ \ell_Q \leq \frac{1}{100}  N_x^{d-1+c} $ for some $Q  $ intersecting $I_x$ then $Q \subset 3 \left(x-[N_x^d,(N_x+N_x^c)^{d}] \right) $. Leveraging the fact that $\inf_J \mathrm{M}f \lesssim \inf_{3J} \mathrm{M}f$, and that $10Q$ must intersect the complement of $\left\{ \mathrm{M}f >c' \sigma \right\}$, it must be that $ \inf_{I_x} \mathrm{M}f  \lesssim  \inf_Q  \mathrm{M}f \lesssim c'  \sigma  $ which cannot happen due to \eqref{eq:bigaverage}. Hence, we are able to infer that $|Q| \geq \frac{1}{100}N_x^{d-1+c}$ whenever $Q \cap  I_x \neq \varnothing. $ Consequently, $I_x \subset 6002^{20d}Q$ which implies that \[ E_{\sigma} \subset \bigcup_{Q \in \mathcal{W}} \left(\left\{x: I_x \subset 6002^{20d}Q,I_x \cap Q \neq \varnothing, N_x^d \leq \ell_Q   \right\} \cup  \left\{x: I_x \subset 6002^{20d}Q,I_x \cap Q \neq \varnothing, N_x^d > \ell_Q  \right\}\right).    \] It is clear that the first set is contained in a $O(1)$ dilate of $Q$. Additionally,    we may notice that the scales that contribute to the second set are $O(1)$ many, independent of Q\footnote{Simply by parametrizing $N $ as $N=N_k={\lambda}^{\lambda^n}$.},  for each fixed $Q \in \mathcal{W}$ hence it is contained in a union of $O(1)$, independent of $Q$, intervals of length controlled by $|Q|$; hence 
        \[\left|\left\{x: I_x \subset 6002^{20d}Q,I_x \cap Q \neq \varnothing, N_x^d > \ell_Q  \right\}\right|  \lesssim |Q|\] therefore \[ |E_{\sigma}| \lesssim \sum_{Q \in \mathcal{W}}|Q| \lesssim \frac{1}{\sigma}  \left \| f \right \|_{L^1(\R)}. \]
	\end{proof}
	With the lemmata above we are ready to give a proof of Proposition \ref{p:varintlp} 
	\begin{proof}
		We start by decomposing $A_Nf=A_N(\Pi_{\leq l_{(N)},\leq -(d-1)L_{(N)}}f)+A_N(f-\Pi_{\leq l_{(N)},\leq -(d-1)L_{(N)}}f).$ Focusing on the major arc piece, we write it as a  sum of its single rational complexity counterparts   \begin{equation} \label{eq:singlecomplexitydecomp}
			A_N(\Pi_{\leq l_{(N)},\leq -(d-1)L_{(N)}}f)=\sum_{l \leq l_{(N)}}A_N(\Pi_{ l,\leq -(d-1)L_{(N)}}f).
		\end{equation} 
		By Proposition \ref{p:majorarcalldeg},	\eqref{eq:vn}, \cite[(6.11)]{KMSZ26} and Riesz-Thorin interpolation (the condition $-(d-1)L_{(N)} \leq -6\max \left\{p,p'\right\}(l+1)$ is evident by the choice of the parameter $\alpha$) we have that \begin{equation} \label{eq:majorarcerror} 
			\left \| A_{N} \left(\Pi_{ l,\leq -(d-1)L_{(N)}}f \right)-T_{\Z} \llbracket G, V_N \eta_{\leq-(d-1)L_{(N)}} \rrbracket_{\Sigma_l} f \right \|_{\ell^p(\Z)} \lesssim N^{-\tau_p} \left \| f \right \|_{\ell^p(\Z)}. 
		\end{equation}
		
		We next verify that at the fixed dyadic complexity scale, the $\ell^p(\Z)$ norm of our variation operator exhibits a power gain   \begin{equation} \label{eq:singlecomplexitygain}
			\left \| \mathcal{V}^{\rho}\left(T_{\Z} \llbracket G, V_N \eta_{\leq -(d-1)L_{(N)}} \rrbracket_{\Sigma_l} f; l_{(N)} \geq l, N \in \mathbb{D}_{\lambda}\right)  \right \|_{\ell^p(\Z)} \lesssim 2^{-\nu_p l} \left \| f \right \|_{\ell^p(\Z)}. 
		\end{equation}
		We apply \cite[Theorem 5.2]{KMSZ26} at dimension $1$ with $N_k=\eps_k^{-1}=\lambda^{\lambda^k}$ which satisfy the hypothesis $0 \leq \eps_k < (4r 2^{l2r})^{-1}$ precisely because of the scale restriction in the variational norm of \eqref{eq:singlecomplexitygain} and $H_1=H_2=\mathbb{C}.$ It remains to show that 
		\begin{equation}
			\begin{split}
				B_p & \coloneqq  \sup_{ \left \| f \right \|_{L^p(\R)} =1 }  \left \|  \mathcal{V}^{\rho}\left( T_{\R} [V_{N_k}\eta_{\leq -(d-1)L_{(N_k)}} ] f ; \ k\in \N \right) \right \|_{ L^p(\R)}
			\end{split}
		\end{equation}
		\begin{equation}
			\begin{split}
				A_{2r} & \coloneqq   \sup_{ \left \| f \right \|_{L^{2r}(\R)} =1 }\sup_{\omega \in \{-1,1\}^{\N}  } \left \| \sum_{k \in \N} \omega(k) T_{\R} \left[ V_{N_{k+1}}\eta_{\leq -(d-1)L_{(N_{k+1})}} - V_{N_k}\eta_{\leq -(d-1)L_{(N_k)}} \right]f  \right \|_{L^{2r}(\R)} 
			\end{split}
		\end{equation}
        are bounded. 
		For $B_p$ we initially note  \[ \begin{split}
			&	 \left(V_{N_k}\eta_{\leq -(d-1)L_{(N_k)}}\right)(0)=1,\\
            &\;\;    \left|V_{N_k}(\xi)\eta_{\leq -(d-1)L_{(N_k)}}(\xi)\right| \lesssim \frac{1}{2^{2^{k}} |\xi| },  \\
            &\; \;  \supp\left(V_{N_k}\eta_{\leq -(d-1)L_{(N_k)}}\right) \subseteq \left\{ \xi: |\xi| \lesssim 2^{-2^k}  \right\}. 
		\end{split} \]  Combining the aforementioned properties and the fact that for all $t>1$, \[ \sup_{ \left \| f \right \|_{L^t(\R)} =1 }  \left \|  \mathcal{V}^{\infty}\left( T_{\R} [V_{N_k}\eta_{\leq -(d-1)L_{(N_k)}} ] f ; \ k\in \N \right) \right \|_{ L^t(\R)} \lesssim 1, \] which follows by Lemma \ref{l:maxfunctionvariant} we have that $B_p \lesssim 1$ by a straightforward adaptation of \cite[Remark 3.134]{KrauseBook} on the real line without the positivity  hypothesis. In regards to the finiteness $A_{2r}$ we will in fact prove that for all $t>1$,  $A_{t}< \infty.$ First, for $t=2$, we show that  \begin{equation}\label{eq:linfty} \sup_{\xi \in \R }\sum_{k}\left|V_{N_{k+1}}\eta_{\leq -(d-1)L_{(N_{k+1})}}(\xi) - V_{N_k}\eta_{\leq -(d-1)L_{(N_k)}}(\xi) \right| \lesssim 1
		\end{equation}  which by Plancherel and trivial $L^{\infty}$ bounds on $V_{N} $ and $\eta$ imply that $A_{2}<\infty.$  Constructing the cancellative bump $\psi(\xi) \coloneqq \eta(2 \xi )-\eta(\xi )$ we can quickly realise that $V_ N$  behaves as the superposition of a non-cancellative bump at scale $N^{-d}$ and a cancellative one at scale $N^{-(d-1)-c}$ up to an error which vanishes at all but finitely many scales for a given frequency $\xi.$ Being more precise, from a moment's reflection using \eqref{eq:vn} one learns that
		\[ \begin{split}
			&	\left|V_N(\xi)-\eta(N^d \xi)-\psi(N^{(d-1)+c}\xi)\right|  \lesssim \cic{1}_{N^{-d}  \lesssim |\xi| \lesssim N^{-(d-1)-c} }+\mathsf{Err}_{N}(\xi)   \\ & \mathsf{Err}_{N}(\xi)   \coloneqq  \max \left\{ \min \left\{ N^{d}|\xi|, (N^{d}|\xi|)^{-\frac{1}{d}}   \right\},\min \left\{ N^{(d-1)+c}|\xi|, (N^{(d-1)+c}|\xi|)^{-1}   \right\}   \right\} .
		\end{split}\]
		It is quite standard that $\displaystyle{\sup_{\xi \in \R} }\left \| \mathsf{Err}_{N}(\xi) \right \|_{\ell^1(\lambda^{\N})} \lesssim 1.$ Furthermore, we can easily see that we have $\displaystyle{\sup_{\xi \in \R} \left \| \cic{1}_{N^{-d}  \lesssim |\xi| \lesssim N^{-(d-1)-c} } \right \|_{\ell^1(\mathbb{D}_{\lambda})}} \lesssim 1$ as for a term of the latter $\ell^1$ norm to be non-zero we must have $N^{-d}  \lesssim |\xi| \lesssim N^{-(d-1)-c} $ which implies that $ \log(N) \sim \log(|\xi|^{-1})$ but since $N \in \mathbb{D}_{\lambda}, $ there is only a finite number of such $N$, independent of $\xi$, that satisfy this condition. We write,	
        \begin{align*}
            &\left|V_{N_{k+1}} \eta_{\leq -(d-1)L_{(N_{k+1})}} - V_{N_k}\eta_{\leq -(d-1)L_{(N_k)}}\right|\\
            &= \left| V_{N_k}-V_{N_{k+1}}+V_{N_k}(\eta_{\le -(d-1) L_{(N_k)}}-1)-V_{N_{k+1}}(-1+\eta_{\le -(d-1) L_{(N_{k+1})}}) \right|.
        \end{align*}      But we have that \[ \left|V_{N_k}(\eta_{\le -(d-1) L_{(N_k)}}-1)\right| \lesssim \frac{1}{N_k^{d-1+c} N_k^{-(d-1)-10\alpha(d-1)} } \lesssim N_k^{-\delta}.  \] 
   Hence, \[\sum_{k \in \N} \left|V_{N_k}(\xi)(\eta_{\le -(d-1) L_{(N_k)}}-1)(\xi)\right| \lesssim 1. \] Furthermore, we have that \[ \begin{split}
			\left|V_{N_k}(\xi)-V_{N_{k+1}}(\xi)\right| & \lesssim \max_{i \in \left\{k,k+1\right\}}\left(\cic{1}_{N_{i}^{-d} \lesssim |\xi| \lesssim N_{i}^{-(d-1)-c} } + \mathsf{Err}_{N_i}(\xi)+\cic{1}_{|\xi| \sim N_i^{-(d-1)-c}}\right)+\cic{1}_{ N_{k+1}^{-d} \lesssim |\xi| \lesssim N_k^{-d} }, 
		\end{split} \] which implies the desired bound \eqref{eq:linfty}, as all error terms introduced are summable over $\mathbb{D}_{\lambda}.$ Once $A_2$ is shown to be finite,   $A_{t}$ will be controlled too for all $t>1$, once we obtain that the kernel of the convolution operator $f \mapsto  f*K_{\omega}$ where \[  K_{\omega}(x) \coloneqq \left( \sum_{k \in \N} \omega(k) \left[V_{N_{k+1}}\eta_{\leq -(d-1)L_{(N_{k+1})}} - V_{N_k}\eta_{\leq -(d-1)L_{(N_k)}} \right] \right)^{\vee}(x)  \] satisfies the H\"{o}rmander condition uniformly in $\omega \in \left\{-1,1\right\}^{\N}$, namely that \[  \sup_{\omega \in  \left\{-1,1\right\}^{\N}} \sup_{h \neq 0} \int_{|x| \geq 2|h|}|K_{\omega}(x-h)-K_{\omega}(x)| \d x  \lesssim 1.   \]
		
		Preliminarily we note the estimate \[ \begin{split}
			&\left \| \Tr_h\left( V_{N_k}^{\vee} * \eta_{\le -(d-1) L_{(N_k)}}^{\vee} \right)- V_{N_k}^{\vee} * \eta_{\le -(d-1) L_{(N_k)}}^{\vee}\right \|_{L^1(\R)}\\
            &  \lesssim  \min\left\{1, \left \| \Tr_hV_{N_k}^{\vee}-V_{N_k}^{\vee}  \right \|_{L^1(\R)} \left \|  \eta_{\le -(d-1) L_{(N_k)}}^{\vee}  \right \|_{L^1(\R)}\right\} \\ & \lesssim  \min \left\{1,\frac{|h|}{N_k^{d-1+c}}\right\}
		\end{split} \] 
		which follows directly after a routine calculation by taking into account that 
		\[ \widehat{V_N}(t)=  \frac{1}{N^cd \left |a_d\right |^{\frac{1}{d}} \left|t\right|^{\frac{1}{d'}}} \cic{1}_{[N^d,(N+N^c)^d]}\left (-\frac{t}{a_d}\right ). \] 
		
		Now, because the regime
		$ N_k \in \left[\frac{1}{100} |h|^{\frac{1}{10d}},  |h|^{\frac{1}{d-1+c}}\right] $ contains $O(1)$ many $k$'s we may dispose of it by using the triangle inequality. 
		
		For the moment we focus on $N_k \leq \frac{1}{100} |h|^{\frac{1}{10d}}.$ We  note the fact that  
		\[ \begin{split}
			&	\int_{|x| \geq 2 |h|} \left| f*g(x-h)-f*g(x) \right| \d x \leq  2 \int_{|x| \geq \frac{1}{10}|h|} \int_{\R}|f|(y) |g|(x-y) \d  x  \d y  
		\end{split}   \]
		and apply it with $ f=V_{N_k}^{\vee}, \; \; g= \eta_{\le -(d-1) L_{(N_k)}}^{\vee}.$ 
		Observe that in the support of $V_{N_k}^{\vee}$ we have that $y \sim N_k^d$ and also because $\eta \in \mathcal{S}$ it holds that \[  \left|(\eta_{\le - (d-1)L_{(N_k)}}^{\vee})(b)\right| \lesssim_M \frac{1}{N_k^{(d-1)(1+10\alpha )}} \left \langle \frac{b}{N_k^{(d-1)(1+10\alpha)}} \right \rangle^{-M}. \] Therefore, \[  \begin{split}
			&	\int_{|x| \geq 2 |h|} \left| V_{N_k}^{\vee}*\eta_{\le -(d-1) L_{(N_k)}}^{\vee}(x-h)-V_{N_k}^{\vee}*\eta_{\le -(d-1) L_{(N_k)}}^{\vee}(x) \right| \d x \\ & \lesssim    \int_{\R}  \frac{1}{y^{1-\frac{1}{d}}N_k^c} \cic{1}_{[N_k^d,(N_k+N_k^c)^d]}(y)   \int_{|t| \geq  \frac{1}{10} |h| } \frac{1}{N_k^{(d-1)(1+10\alpha)}} \left \langle \frac{t-y}{N_k^{(d-1)(1+10\alpha)}} \right \rangle^{-M} \d t  \d y   \lesssim N_k^{-10d},
		\end{split}  \]
		where for the passage to the last inequality is possible because $ |h| \geq 100^{10d} N_k^{10d}.$ We may put these estimates together to obtain, \[ \begin{split}
		&	\sup_{\omega \in  \left\{-1,1\right\}^{\N}}\int_{|x| \geq 2|h|}|K_{\omega}(x-h)-K_{\omega}(x)| \d x  \\ &\leq \sum_{k \in \N} 	\left \| \Tr_h\left( V_{N_k}^{\vee} * \eta_{\le - (d-1)L_{(N_k)}}^{\vee} \right)- V_{N_k}^{\vee} * \eta_{\le -(d-1) L_{(N_k)}}^{\vee}\right \|_{L^1(|x| \geq 2 |h| )} \\ &  \lesssim  \sum_{N_k < \frac{1}{100}|h|^{\frac{1}{10d}} }N_k^{-10d}+\sum_{N_k \in  \left[\frac{1}{100} |h|^{\frac{1}{10d}},  |h|^{\frac{1}{d-1+c}}\right] }1+ \sum_{N_k > |h|^{\frac{1}{d-1+c}} }\frac{|h|}{N_k^{d-1+c}} \lesssim 1.
		\end{split}  \]
		As in \cite{KMSZ26} we may factorize the multiplier and write  \[ T_{\Z}\llbracket G , V_N \eta_{\leq -(d-1)L_{(N)}} \rrbracket_{\Sigma_{  l}}= T_{\Z}  \llbracket 1, V_N \eta_{\leq -(d-1)L_{(N)}} \rrbracket_{\Sigma_{ \leq l}}  T_{\Z}\llbracket G, \eta_{\leq -(d-1)200(l+1)r} \rrbracket_{\Sigma_l}    \] and therefore \[ \left \| T_{\Z}\llbracket G , V_N \eta_{\leq -(d-1)L_{(N)}} \rrbracket_{\Sigma_{  l}}f \right \|_{\ell^2(\Z)}  \lesssim 2^{-\chi l}  \left \| f \right \|_{\ell^2(\Z)}  \] for some $\chi>0 $, Riesz-Thorin interpolation and the conclusion of \cite{KMSZ26}[Theorem 5.2] yields \eqref{eq:singlecomplexitygain}. A combination of \eqref{eq:majorarcerror},\eqref{eq:singlecomplexitydecomp} and \eqref{eq:singlecomplexitygain} yields \[\begin{split}
			\left \| \mathcal{V}^{\rho} \left( A_N \left(  \Pi_{\leq l_{(N)},\leq -(d-1)L_{(N)}} f \right) \right) \right \|_{\ell^p(\Z)} & \leq  \sum_{l \in \N} \left \| \mathcal{V}^{\rho}\left(T_{\Z} \llbracket G, V_N \eta_{\leq -(d-1)L_{(N)}} \rrbracket_{\Sigma_l} f; l_{(N)} \geq l, N \in \mathbb{D}_{\lambda}  \right) \right \|_{\ell^p(\Z)} \\ & +\sum_{l \in \N} \sum_{\substack{N \in \mathbb{D}_{\lambda} \\ l_{(N)} \geq l  }}	\left \| A_{N} \left(\Pi_{ l,\leq -(d-1)L_{(N)}}f \right)-T_{\Z} \llbracket G, V_N \eta_{\leq -(d-1)L_{(N)}} \rrbracket_{\Sigma_l} f \right \|_{\ell^p(\Z)} \\ &  \lesssim \left(\sum_{l \in \N} 2^{-\nu_p l}+\sum_{N \in \mathbb{D}_{\lambda}} l_{(N)} N^{-\tau_p}\right) \left \| f \right \|_{\ell^p(\Z)} \lesssim \left  \| f \right \|_{\ell^p(\Z)}. 
		\end{split}    \]
		We complete the proof by addressing the minor arc piece,  \begin{equation} \label{eq:minorarcerror}
			\left \| A_N(f-\Pi_{\leq l_{(N)},\leq -(d-1)L_{(N)}}f) \right \|_{\ell^p(\Z)} \lesssim N^{-\tau_p} \left \| f \right \|_{\ell^p(\Z)}
		\end{equation} which is immediate after combining \cite[(6.11)]{KMSZ26}, Lemma \ref{l:minorarcalldegree}, \eqref{eq:vn} in conjunction with proposition \ref{p:majorarcalldeg}  and the fact that $A_N$ is contraction on $\ell^p(\Z).$ In an immediate fashion, we have that \[ \left \| \mathcal{V}^{\rho} \left(A_N(f-\Pi_{\leq l_{(N)},\leq -(d-1)L_{(N)}}f) \right) \right \|_{\ell^p(\Z)} \lesssim \sum_{N \in \mathbb{D}_{\lambda}}	\left \| A_N(f-\Pi_{\leq l_{(N)},\leq -(d-1)L_{(N)}}f) \right \|_{\ell^p(\Z)} \lesssim \left \| f \right \|_{\ell^p(\Z)}.  \]
	\end{proof}
\section{Failure of pointwise convergence at $L^1$}

This section is devoted to proving that no universal pointwise ergodic theorem
holds in \(L^1\) for short-interval monomial averages of degree \(d\geq2\),
along any prescribed infinite subsequence of integers. In particular, the result
applies to every lacunary subsequence. We follow LaVictoire's paradigm
\cite{L9}, with the necessary modifications to account for the short-interval
nature of the averages. Throughout the section, \(\mathcal N\subseteq\mathbb N\)
is fixed, infinite, and otherwise arbitrary, and we set \(H_N \coloneqq  \floor{N^c}\) the short
interval's length. 

More precisely, the main task of this section is to prove the following finite
cyclic weak-type failure.
\begin{proposition}\label{p:finite-weak-failure}
	For every \(C>0\), there are a positive integer \(T\), the cyclic system
	\(\mathbb Z_T\) with shift \(x\mapsto x+1\), and a non-negative function
	\(f\in \ell^1(\mathbb Z_T)\) such that
	\[
	\left\|
	\sup_{N\in\mathcal N}
	\left|
A_{N}^{t^d,c}f(x)
	\right|
	\right\|_{\ell^{1,\infty}(\mathbb Z_T)}
	>
	C\|f\|_{\ell^1(\mathbb Z_T)}.
	\]
	Here \(\ell^1(\mathbb Z_T)\) and \(\ell^{1,\infty}(\mathbb Z_T)\) are taken with
	respect to normalized counting measure.
\end{proposition}
Proposition \ref{p:finite-weak-failure} can readily be seen to imply 
Theorem~\ref{t:short-l1-bad}  by the standard Sawyer-Conze reduction used
in \cite{BM,L9}. The proof's architecture has two components. The first component is intrinsic to 
LaVictoire's inductive construction \cite[Section~4]{L9} which we recycle. The output is a family of non-negative functions \(f_h\),
auxiliary random variables \(X_h\), and stopping time, $Q_x$,
with the property that
\[
\frac1{|\Lambda_{Q_x}|}
\sum_{a\in\Lambda_{Q_x}} f_h(x+a)
\geq X_h(x).
\]
The random variables $X_h$'s are pairwise independent and have large mean, while the
$f_h$'s have controlled \(L^1\)-norm. The second
component is external: we establish a lower bound for the
short-interval maximal operator which is exactly where our setting differs from LaVictoire's.

\bigskip

Following LaVictoire we define the residue sets used throughout the argument. To this end, due to Dirichlet's theorem on primes in arithmetic progressions we may choose an infinite
set of primes \(\mathcal P_d\), all congruent to \(1\pmod d\), so sparse that $\sum_{p\in\mathcal P_d}p^{-1} \leq \frac{1}{10}$ and therefore

\begin{equation} \label{eq:sparseprimes}
	\prod_{p\in\mathcal P_d}\left(1-\frac1p\right)\geq \frac12.
\end{equation}

Let \(\mathcal Q\) denote the set of squarefree products of primes from
\(\mathcal P_d\). For \(Q\in\mathcal Q\), define
\[
\Lambda_Q
:=
\{a\in \mathbb Z/Q\mathbb Z:
(a,Q)=1
\text{ and } a\equiv r^d\pmod Q
\text{ for some } r\in\mathbb Z\}.
\]
That is,  $\Lambda_Q$ is the set of invertible $d$-th-power residues modulo $Q$.

We record the properties that will replace LaVictoire's use of the
ordinary initial averages.
\begin{lemma}\label{l:root-count}
Let \(Q\in\mathcal Q\), and write \(\omega(Q)\) for the number of prime factors
of \(Q\). Then
\[
|\Lambda_Q|=\frac{\varphi(Q)}{d^{\omega(Q)}}.
\]
Moreover, every \(a\in\Lambda_Q\) has exactly \(d^{\omega(Q)}\) \(d\)-th roots
modulo \(Q\); that is,
\[
\#\{r\in\mathbb Z/Q\mathbb Z:r^d\equiv a\pmod Q\}
=
d^{\omega(Q)}.
\]
\end{lemma}

\begin{remark}
The set \(\Lambda_Q\) is the reduced \(d\)-th-power residue set modulo \(Q\),
namely the image of \(r\mapsto r^d\) on
\((\mathbb Z/Q\mathbb Z)^\times\). This is the same convention used in
LaVictoire's construction, where the Granville--Kurlberg spacing input is
applied to reduced residue sets.
\end{remark}

\begin{proof}[Proof of Lemma \ref{l:root-count}]
	First suppose $Q=p\in\mathcal P_d$. Since $p\equiv1\pmod d$, the group
	$(\mathbb Z/p\mathbb Z)^\times$ is cyclic of order divisible by $d$. The map $	r\mapsto r^d$ is a homomorphism from \((\mathbb Z/p\mathbb Z)^\times\) to itself. Its kernel
	has exactly \(d\) elements, and therefore its image has cardinality \((p-1)/d\).
	Consequently every invertible \(d\)-th-power residue modulo \(p\) has exactly
	\(d\) roots.
	
	For squarefree \(Q=p_1\cdots p_\kappa\), the Chinese remainder theorem gives
	\[
	(\mathbb Z/Q\mathbb Z)^\times
	\simeq
	\prod_{i=1}^{\kappa}(\mathbb Z/p_i\mathbb Z)^\times.
	\]
	The \(d\)-th-power map acts coordinatewise. Hence, its kernel has size
	\(d^\kappa\), and its image has size $
	\frac{\varphi(Q)}{d^\kappa}$
and the proof is complete.
\end{proof}The next lemma is the key short-interval replacement. It says that once the
short interval has length much larger than \(Q\), it witnesses every residue
\(a\in\Lambda_Q\) with a uniform lower bound.

\begin{lemma}[short-block residue count]\label{l:short-block-count}
Let \(Q\in\mathcal Q\), let \(U\in\mathbb Z\), and let \(H\geq 4Q\). Then,
for every \(a\in\Lambda_Q\),
\[
\#\{1\leq m\leq H:(U+m)^d\equiv a\pmod Q\}
\geq
\frac{H}{4|\Lambda_Q|}.
\]
\end{lemma}

\begin{proof}
Note that
 \[
\begin{split}
	 &\#\left\{1\leq m\leq H:(U+m)^d\equiv a\pmod Q\right \} \\
     & \geq \floor{\frac{H}{Q}} \# \left\{m \in \Z / Q \Z : m^d \equiv a \pmod{Q}\right\}   \geq \frac{H}{2Q}d^{\omega(Q)} \ge \frac{H}{4\left|\Lambda_Q\right|},
\end{split} \] where the passage to the last inequality is possible to due to  Euler product identity \[\frac{\varphi(Q)}{Q}=\prod_{p \mid Q} \left(1-\frac{1}{p}\right)\] and \eqref{eq:sparseprimes}.
\end{proof}

\begin{remark}
Lemma \ref{l:short-block-count} counts only the residue classes in the reduced set \(\Lambda_Q\).
This is in our application.  Indeed, the averages will be along
\((U+m)^d\), which may also fall into non-invertible residue classes modulo
\(Q\). However, all functions are non-negative, so those terms may simply be
discarded.  As we shall use below, the prime factors of \(Q\) are chosen from a sparse set
\(\mathcal P_d\) for which $\varphi(Q)/Q \ge  1/2$.
Thus restricting to reduced residue classes changes the lower bound in
Lemma~\ref{l:short-block-count} only by an absolute constant.
\end{remark}

\bigskip

We now prepare the control function which will replace LaVictoire's
initial-average control function \(\psi\).
We refine the family of moduli from the previous section so
that it is compatible with LaVictoire's construction for \(d\)-th powers.

Choose pairwise coprime integers
\[
p_1,p_2,\dots,
\qquad
q_1,q_2,\dots,
\]
where each \(p_j\) is a prime in \(\mathcal P_d\), and each \(q_j\) is a
squarefree product of primes from \(\mathcal P_d\), with $p_i,\ q_j$
pairwise coprime as integers and with
\[
\omega(q_j)\to\infty .
\]
This can be done by choosing disjoint finite subsets of the sparse prime set
\(\mathcal P_d\). For every squarefree product \(Q\) formed from the \(p_j\)'s
and \(q_j\)'s, its prime factors are still contained in \(\mathcal P_d\). Henceforth, standard Euler product identities and elementary inequalities yield
\[
\frac{\varphi(Q)}{Q}
=
\prod_{p\mid Q}\left(1-\frac1p\right)
\geq
\prod_{p\in\mathcal P_d}\left(1-\frac1p\right)
\geq
\frac12.
\]
In particular,
\[
\frac{Q}{\varphi(Q)}\leq 2.
\]

For the rest of the proof, we restrict \(\mathcal Q\) to the subfamily of
squarefree products
\[
Q=p_{i_1}\cdots p_{i_k}q_{j_1}\cdots q_{j_\ell},
\qquad
i_1<\cdots<i_k,\quad j_1<\cdots<j_\ell .
\]
We continue to denote this restricted family by \(\mathcal Q\). This is the
product family used in \cite[Proposition~4.1]{L9}. Since it is a subfamily of
the moduli from the previous section, all residue-counting lemmas remain valid
for every \(Q\in\mathcal Q\).

For each \(Q\in\mathcal Q\), choose \(N(Q)\in\mathcal N\) so large that
\[
\floor{N(Q)^c}\geq 4Q
.
\]
This is possible because \(\mathcal N\) is infinite and hence unbounded.

Define the short-interval control function \(\Psi\) by
\[
\Psi(u)
:=
\ceiling{u}
+
\max_{\substack{Q\in\mathcal Q\\ Q\leq u}}
\bigl(N(Q)+ \floor{N(Q)^c} \bigr)^d,
\qquad u\geq 1,
\]
where the maximum over the empty set is interpreted as \(0\). Thus \(\Psi\) is
finite and non-decreasing. Moreover, for every \(Q\in\mathcal Q\), every
\(A\geq1\), and every \(1\leq m\leq H_{N(Q)}\), we have
\[
Q
\leq
\bigl(N(Q)+m\bigr)^d
\leq
\bigl(N(Q)+\floor{N(Q)^c}\bigr)^d
\leq
\Psi(Q)
\leq
\Psi(AQ).
\]
This is the only role of \(\Psi\): it records how far local \(Q\)-periodicity
must persist in order to cover the short interval attached to \(Q\).

We next state the precise form of LaVictoire's inductive output which we shall
use. The proposition below is the terminal case \(L=M\) of
\cite[Proposition~4.1]{L9}, adapted to the present short-interval setting by
replacing LaVictoire's control function \(\psi\) with the function \(\Psi\)
defined above. We do not reproduce the full leakage construction. However, since
the replacement of \(\psi\) by \(\Psi\) is the only place where the present
argument differs from LaVictoire's, we verify that the
induction in \cite[Sections~4--8]{L9} is stable under this replacement.

First recall LaVictoire's auxiliary distribution. For \(0<\gamma<1\) and
\(0<\alpha<1\), define random variables \(Y_{n,\gamma,\alpha}\) recursively by
\[
Y_{0,\gamma,\alpha}\equiv 1,
\]
and
\[
Y_{n+1,\gamma,\alpha}(x_0,\dots,x_{n+1})
:=
(1-\gamma)^{-1}
Y_{n,\gamma,\alpha}(x_0,\dots,x_n)
1_{[\gamma,1)}(x_{n+1})
+
\alpha 1_{[0,\alpha\gamma)}(x_{n+1}).
\]
Thus
\[
\mathbb E Y_{n,\gamma,\alpha}
=
1+n\alpha^2\gamma,
\]
and, for the range of \(\gamma\) used in the construction,
\[
\mathbb E Y_{n,\gamma,\alpha}^2
\leq
2(1-\gamma)^{-n}.
\]

\begin{proposition}
	\label{p:lavictoire-short-control}
	Fix \(0<\gamma<\gamma_0\), with \(\gamma\) dyadic, and let \(K,M\in\mathbb N\).
	Let \(A\geq1\), let $\delta=o(\gamma)$, and let \(D\) be an odd integer. Then there
	exist a finite cyclic group \(\mathbb Z_{\mathcal T}\), non-negative functions $	f_1,\dots,f_K:\mathbb Z_{\mathcal T}\to[0,\infty),$ random variables $	X_1,\dots,X_K:\mathbb Z_{\mathcal T}\to[0,\infty),$ an exceptional set \(E\subseteq\mathbb Z_{\mathcal T}\), and a map $	Q_x:\mathbb Z_{\mathcal T}\to\mathcal Q$ such that the following properties hold:  \begin{itemize}
		\item $	\mathbb P(E)\leq \delta $
		\item  $	1\leq \mathbb E f_h\leq (1+4\varepsilon)^{KM},$ where \(\varepsilon=\varepsilon_\gamma=o(\gamma)\) is chosen sufficiently small and $100\delta < \varepsilon$, for each \(1\leq h\leq K\);
		\item The random variables \(X_1,\dots,X_K\) are pairwise independent and
		\[
		X_h\stackrel{d}{=}Y_{M,\gamma,\alpha}
		\qquad
		\text{for every }1\leq h\leq K;
		\]
		\item for every \(x\notin E\), the modulus \(Q_x\) divides \(\mathcal T\), and
		for every \(1\leq h\leq K\),
		\begin{align}\label{e:rdomest}
		    \frac1{|\Lambda_{Q_x}|}  \sum_{\substack{ a\in\Lambda_{Q_x}\\1\leq a\leq Q_x}}
		f_h(x+a)
		\geq
		X_h(x);
		\end{align}
		\item  the functions \(f_h\) are locally \(Q_x\)-periodic near \(x\) up to the
		short-interval control scale: 	
        \begin{align}\label{e:LPIdentity}
            f_h(x+y-Q_x)=f_h(x+y) \text{ whenever } 	1\leq h\leq K,
		\
		x\notin E,
		\; 
		Q_x\leq y\leq \Psi(AQ_x).
        \end{align}

	\end{itemize}

\end{proposition}

\begin{proof}
	We prove the proposition by adapting LaVictoire's proof of
\cite[Proposition~4.1]{L9}. LaVictoire's Proposition does not directly apply to the
present short-interval averages. What we use is the robustness of his inductive
construction: the function \(\psi\) appears only as the finite range in the
local-periodicity condition. Since our function \(\Psi\) is also finite and
non-decreasing, the same construction works with \(\Psi\) in place of \(\psi\).
We now check this replacement in the induction.

We now fix the constant \(\alpha\) used in the auxiliary distribution
\(Y_{n,\gamma,\alpha}\). By the Granville--Kurlberg spacing theorem, in the
form used in \cite[(3.4)]{L9}, we may choose constants
\(0<\alpha<1\) and \(0<\gamma_0<1\) such that, for every dyadic
\(0<\gamma<\gamma_0\), the auxiliary moduli \(q=q_j\) may be taken sufficiently
far out so that
\[
\frac1{|\Lambda_q|^2}
\#\Bigl\{
(u,v)\in\Lambda_q^2:
\operatorname{dist}(u-v,\Lambda_q)>\gamma s_q
\Bigr\}
>
5\alpha .
\tag{GK}
\]
It is independent of
\(K,M,A,\delta,D\); only the choice of how far out one takes \(q=q_j\) may depend
on these later parameters.

Let us write \(\operatorname{Step}_{\Psi}(K,M,L)\) for the statement of
\cite[Proposition~4.1]{L9} at the internal stage \(L\), but with the
local-periodicity conclusion
\[
f_h(x+y-Q_x)=f_h(x+y),
\qquad
Q_x\leq y\leq \psi(AQ_x),
\]
replaced by
\[
f_h(x+y-Q_x)=f_h(x+y),
\qquad
Q_x\leq y\leq \Psi(AQ_x).
\]
All other conclusions are kept as in LaVictoire's stage-\(L\) statement. In
particular, the last variable has distribution \(Y_{L,\gamma,\alpha}\), and the
stage-\(L\) \(\ell^1\)-bound is
\[
\mathbb E f_h\leq (1+4\varepsilon)^{(K-1)M+L}.
\]
Thus the required proposition is the terminal case
\(\operatorname{Step}_{\Psi}(K,M,M)\).
	
It remains to check the induction step. Assume that \(L<M\), that
\(\operatorname{Step}_{\Psi}(K,M,L)\) is known, and that all previous steps are
known for all parameter values \(A,\delta,D\), with \(\gamma\) fixed. Fix
\(A,\delta,D\). As in \cite[Section~8]{L9}, apply
\(\operatorname{Step}_{\Psi}(K,M,L)\) with parameters \(A,\delta/4,D\). In the
inductive notation, the ambient cyclic period \(\mathcal T\) at stage \(L\) is
written as \(T_LR_L\). Thus we obtain
\[
T_L,\ R_L,\ f_1^L,\dots,f_K^L,\ X_1^L,\dots,X_K^L,\ E_L,\ Q_{x,L}.
\]
	The construction then chooses a new prime \(p\) and a new auxiliary modulus
\(q\), and performs the \(p\)-periodic rearrangement from
\cite[Section~5]{L9}. This produces the rearranged family
\[
\widetilde f_1^L,\dots,\widetilde f_K^L,
\qquad
\widetilde X_1^L,\dots,\widetilde X_K^L,
\qquad
\widetilde Q_{x,L}.
\]
It then introduces the leakage sets \(\Delta_q\) and \(\Psi_q\) from
\cite[Section~6]{L9}. We recall their definitions and properties for ease of
reference. Here \(\Psi_q\) is LaVictoire's leakage set and should not be confused
	with the global control function \(\Psi\).
	Write
	\[
	s_q:=\mathbb P(\Lambda_q)^{-1}=\frac{q}{|\Lambda_q|}.
	\]
	The following set is useful in the proof: 
	\[
	\Phi_q
	:=
	\left\{
	u\in-\Lambda_q:
	\#\left\{
	v\in\Lambda_q:
	u+v\notin -\Lambda_q+(-\gamma s_q,\gamma s_q)
	\right\}
	\geq
	2\alpha|\Lambda_q|
	\right\}.
	\]
	By the spacing theorem of Granville--Kurlberg \cite[Corollary 2]{GK08}, in the form used by
	LaVictoire in \cite[(3.4)]{L9}, the auxiliary moduli \(q=q_j\) may be chosen so
	that the normalized spacings of \(\Lambda_q\) are sufficiently close to the
	Poisson model. In particular, for the present value of \(\alpha\), we may choose
	\(q\) so large that$	|\Phi_q|\geq 3\alpha|\Lambda_q|.$ Then set $	\Phi_q^\gamma:=\Phi_q+(0,\gamma s_q)\subseteq (-\Lambda_q)^\gamma .$ The sets \(\Psi_q,\Delta_q\subseteq\mathbb Z_q\) are chosen so that the following properties are satisfied:
	
\begin{itemize}
	 \item $\Psi_q\subseteq\Phi_q^\gamma\setminus\Delta_q,$
	\item
	$\mathbb P(\Psi_q)\geq \alpha\gamma,$
	
\item 
$	\mathbb Z_q\setminus(-\Lambda_q)^\gamma
	\subseteq
	\Delta_q
	\subseteq
	\mathbb Z_q\setminus\Phi_q,
	$

\item
$	\mathbb P(\Delta_q)
	<
	1-\mathbb P(\Lambda_q^\gamma)+\delta/8,
	$
\item 

$	(|\Psi_q|,D)=(|\Delta_q|,D)=1.
$

\end{itemize}
	The last coprimality condition can be imposed because \(q\) is chosen much
	larger than \(D\). The key consequence of these definitions is that
	\[
	\#\{v\in\Lambda_q:x+v\in\Delta_q\}
	\geq
	2\alpha|\Lambda_q|
	\qquad
	\text{for every }x\in\Psi_q.
	\]
	These choices and definitions are unchanged in the present argument, except for
	one harmless strengthening of the choice of \(q\), described below.
	
	The only places before \cite[Section~8]{L9} where the numerical size of \(\psi\) enters are
	the exceptional sets denoted $E_L^1$ in \cite[(5.2)]{L9} and  \(E_L^2\) in \cite[(6.9)]{L9}. In the present
	argument we define instead
	\[\begin{split}& E_L^1(\omega):= \tilde E^\omega \cup\bigcup_{0\leq i\leq\frac{\sqrt p}{T_L} -1} [iT_L\lfloor\sqrt p\rfloor -\Psi(AT_L),iT_L\lfloor\sqrt p\rfloor]+p\Z
	\\ &	E_L^2
\coloneqq 
	\{x\in\Delta_q:\exists\,0<y\leq \Psi(AQ_{x,L})
	\text{ such that }x+y\notin\Delta_q\}.
\end{split}
	\]
	Since \(Q_{x,L}\mid T_L\), we have \(Q_{x,L}\leq T_L\). Since \(\Psi\) is
	non-decreasing $	\Psi(AQ_{x,L})\leq \Psi(AT_L).$ 	Hence,
	\[
	E_L^2
	\subseteq
	\bigl(\Delta_q\cap(-\Lambda_q)^\gamma\bigr)
	\cup
	\bigl(-\Lambda_q+(-\Psi(AT_L),0]\bigr).
	\]
	At this stage \(T_L\) has already been fixed, and therefore \(\Psi(AT_L)\) is a
	finite number. From the inclusion above we get
	\[
	\mathbb P(E_L^2)
	\leq
	\mathbb P\bigl(\Delta_q\cap(-\Lambda_q)^\gamma\bigr)
	+
	\mathbb P\bigl(-\Lambda_q+(-\Psi(AT_L),0]\bigr).	\]
	The first term is estimated exactly as in \cite[(6.9)]{L9}; by choosing \(q\)
	as in LaVictoire's argument, it is at most \(\delta/8\). Indeed,
    $$\mathbb P\bigl(\Delta_q\cap(-\Lambda_q)^\gamma\bigr) = \mathbb P\bigl(\Delta_q\bigr)+\mathbb P\bigl((-\Lambda_q)^\gamma\bigr) - \mathbb P\bigl(\Delta_q\cup(-\Lambda_q)^\gamma\bigr) < 1+\delta/8  - \mathbb P\bigl(\Delta_q\cup(-\Lambda_q)^\gamma\bigr)= \delta/8.$$
    For the second term,
	the set
	\[
	-\Lambda_q+(-\Psi(AT_L),0]
	\]
	is contained in the union of at most \(\lfloor \Psi(AT_L)\rfloor+1\) translates
	of \(-\Lambda_q\). Hence
	\[
	\mathbb P\bigl(-\Lambda_q+(-\Psi(AT_L),0]\bigr)
	\leq
	(\lfloor \Psi(AT_L)\rfloor+1)\mathbb P(\Lambda_q).
	\]
	Since the auxiliary moduli \(q=q_j\) may be chosen arbitrarily far out in
	LaVictoire's sequence, and since \(\mathbb P(\Lambda_q)\to0\), we may choose
	\(q\) so large that
	\[
	(\lfloor \Psi(AT_L)\rfloor+1)\mathbb P(\Lambda_q)
	\leq
	\frac{\delta}{8}.
	\]
	Therefore
	\[
	\mathbb P(E_L^2)\leq \frac{\delta}{8}+\frac{\delta}{8}
	\leq
	\frac{\delta}{4}. 
	\] 
    Similarly, as in \cite[Section~5]{L9}, after choosing \(p\) sufficiently large
we have
\[
\mathbb P(E_L^1(\omega))
\leq
\mathbb P(E_L)
+
\frac{\Psi(AT_L)+2T_L}{T_L\sqrt p}.
\]
Since the previous-stage family was chosen with exceptional-set parameter
\(\delta/4\), we have $\mathbb P(E_L)\leq \frac{\delta}{4}.$
Choosing \(p\) large enough so that
\[
\frac{\Psi(AT_L)+2T_L}{T_L\sqrt p}
\leq
\frac{\delta}{4},
\]
we obtain
\[
\mathbb P(E_L^1(\omega))\leq \frac{\delta}{2}.
\]
We now use the restriction step from \cite[Lemma~7.1]{L9}. Put $A_L:=AT_Lpq$ and $D_L:=DT_Lpq.$
By the strong inductive hypothesis, take a
\(\operatorname{Step}_{\Psi}(K-1,M,M)\) family with parameters
\(A_L,\delta/4,D_L\). Denote this family by
\[
g_1,\dots,g_{K-1},
\qquad
Z_1,\dots,Z_{K-1},
\qquad
E',
\qquad
Q'_x.
\]
Applying \cite[Lemma~7.1]{L9} with \(B=T_Lp\) restricts this family to
\((-\Lambda_q)^\gamma\). We obtain
\[
\bar g_1,\dots,\bar g_{K-1},
\qquad
\bar Z_1,\dots,\bar Z_{K-1},
\qquad
\bar E',
\qquad
\bar Q'_x.
\]
The conclusions of \cite[Lemma~7.1]{L9} give
\[
\mathbb P(\bar E')\leq \delta/4,
\]
and, for every \(1\leq h\leq K-1\),
\[
\mathbb E\bar g_h
\leq
\gamma\,\mathbb E g_h
\leq
\gamma(1+4\varepsilon)^{(K-1)M}
\leq
\gamma(1+4\varepsilon)^{(K-1)M+L}.
\]
The same lemma gives the conditional distribution identities for the variables
\(\bar Z_h\) on \(q\)-periodic subsets of \((-\Lambda_q)^\gamma\), in the form
used in the recombination argument \cite[(8.7)]{L9}\footnote{We do not need, and do not
claim, global pairwise independence of the restricted variables.}.  Finally, the
restriction lemma gives the residue domination estimate
\[
\frac1{|\Lambda_{\bar Q'_x}|}
\sum_{a\in\Lambda_{\bar Q'_x}}
\bar g_h(x+a)
\geq
\bar Z_h(x),
\qquad
x\notin\bar E',
\quad
1\leq h\leq K-1.
\]

We only need to check that the local-periodicity conclusion remains valid with
\(\Psi\) in place of LaVictoire's \(\psi\). In the present application, the
input family for \cite[Lemma~7.1]{L9} has parameter $A_L:=AT_Lpq,$
and the lemma is applied with \(B=T_Lp\). As in LaVictoire's notation, the
restricted family then has output parameter $\bar A_L=A.$

The new stopping modulus is
\[
\bar Q'_x=qBQ'_x=qT_Lp\,Q'_x.
\]
In particular, $pqT_L\mid \bar Q'_x.$
Moreover, $\bar A_L\bar Q'_x
=
AqT_Lp\,Q'_x
=
A_LQ'_x.$
Therefore the upper endpoint in the local-periodicity range is unchanged:
\[
\Psi(\bar A_L\bar Q'_x)
=
\Psi(A_LQ'_x).
\]
Thus the proof of \cite[Lemma~7.1]{L9} applies with \(\Psi\) exactly as it does
with \(\psi\): the lemma only uses the fact that the same finite control range is
preserved under the parameter change. Hence the restricted family satisfies the
same conclusions as in \cite[Lemma~7.1]{L9}, with \(\Psi\) in the
local-periodicity condition.
	
	Let \(S\) be the auxiliary integer chosen in \cite[Section~8]{L9}. Define,
exactly as in \cite[(8.1)--(8.5)]{L9},
\begin{align}
    T_{L+1}&:=ST_Lpq,\\
	E_{L+1}&:=E_L^1(\omega)\cup E_L^2\cup \bar E',\\
	Q_{x,L+1}
	&:=
	\begin{cases}
		\widetilde Q_{x,L},& x\in\Delta_q,\\
		\bar Q'_x,& x\notin\Delta_q,
	\end{cases}
\end{align}
	and, for \(1\leq h\leq K-1\),
	\begin{align}\label{e:L+1def}
	    f_h^{L+1}
	:=
	\widetilde f_h^L1_{\Delta_q}
	+
	\bar g_h1_{\mathbb Z\setminus\Delta_q},
	X_h^{L+1}
	:=
	\widetilde X_h^L1_{\Delta_q}
	+
	\bar Z_h1_{\mathbb Z\setminus\Delta_q}.
	\end{align}
	Also set
	\[
	f_K^{L+1}:=(1-\gamma)^{-1}\widetilde f_K^L1_{\Delta_q}.
	\]
	The definition of \(X_K^{L+1}\) is exactly the one in \cite[Section~8]{L9}: one
	uses the auxiliary sets \(\Gamma_s,\Gamma_t\) from \cite[(8.6)]{L9} to force the
	exact distribution
	\[
	X_K^{L+1}\stackrel{d}{=}Y_{L+1,\gamma,\alpha}\coloneq 
(1-\gamma)^{-1}\widetilde X_K^L1_{\Delta_q}1_{\Gamma_s}
+\alpha1_{\Psi_q}1_{\Gamma_t}.
	\]
	We check the four properties of \(\operatorname{Step}_{\Psi}(K,M,L+1)\).

\medskip
    
	First, the \(\ell^1\)-bounds for \(f_h^{L+1}\) are unchanged. For
	\(1\leq h\leq K-1\), LaVictoire's estimate gives
	\[
	\mathbb E f_h^{L+1}
	\leq
	(1-\gamma+2\varepsilon)\mathbb E f_h^L
	+
	\gamma\mathbb E g_h
	\leq
	(1+4\varepsilon)^{(K-1)M+L+1}.
	\]
	For \(h=K\), the estimate follows from the definition
	\[
	f_K^{L+1}=(1-\gamma)^{-1}\widetilde f_K^L1_{\Delta_q}
	\]
	and the same estimate as in \cite[Section~8, verification of property~(1)]{L9}.
	These estimates do not use the special form of \(\psi\).

\bigskip
    
	Second, the distribution and pairwise independence of
\(X_1^{L+1},\dots,X_K^{L+1}\) are verified exactly as in
\cite[Section~8, verification of property~(2)]{L9}. The construction splits the
space into the pieces \(\Delta_q\), \(\Psi_q\), and the remaining region, and on
each piece the conditional distributions are the same as in LaVictoire's
argument. The recombination identity \cite[(8.7)]{L9} then gives the desired
pairwise independence. This part of the proof does not use the control function
\(\psi\), and hence is unchanged with \(\Psi\).

\bigskip

Third, the exceptional set satisfies
\[
E_{L+1}
=
E_L^1(\omega)\cup E_L^2\cup \bar E',
\]
and therefore
\[
\mathbb P(E_{L+1})
\leq
\mathbb P(E_L^1(\omega))
+
\mathbb P(E_L^2)
+
\mathbb P(\bar E').
\]
By the estimates above,
\[
\mathbb P(E_L^1(\omega))\leq \frac{\delta}{2},
\qquad
\mathbb P(E_L^2)\leq \frac{\delta}{4},
\qquad
\mathbb P(\bar E')\leq \frac{\delta}{4}.
\]
The first estimate is obtained by choosing \(p\) sufficiently large, the second
by choosing \(q\) sufficiently large so that the boundary term involving
\(\Psi(AT_L)\) is small, and the third is part of the restriction lemma
\cite[Lemma~7.1]{L9}. Hence
\[
\mathbb P(E_{L+1})\leq
\frac{\delta}{2}
+
\frac{\delta}{4}
+
\frac{\delta}{4}
=
\delta.
\]

\bigskip

Fourth, we verify residue domination and local periodicity. Fix
\(x\notin E_{L+1}\), and put \(Q:=Q_{x,L+1}\). We consider separately
whether \(x\in\Delta_q\) and whether \(h<K\).

\medskip

Suppose first that \(x\in\Delta_q\). Then
\(Q=\widetilde Q_{x,L}\). Since \(x\notin E_L^2\), we have $x+t\in\Delta_q$ whenever $0\leq t\leq\Psi(AT_L).$
In particular, this holds for every representative
\(a\in\Lambda_Q\), since \(1\leq a\leq Q\leq T_L\), and also for
\(t=y-Q\) and \(t=y\) whenever
\[
Q\leq y\leq\Psi(AQ)\leq\Psi(AT_L).
\]
The exclusion of \(E_L^1(\omega)\) ensures that the rearranged family
\[
\widetilde f_1^L,\dots,\widetilde f_K^L,\qquad
\widetilde X_1^L,\dots,\widetilde X_K^L,\qquad
\widetilde Q_{x,L}
\]
satisfies the residue-domination and local-periodicity conclusions inherited
from \(\operatorname{Step}_{\Psi}(K,M,L)\).

If \(1\leq h\leq K-1\), the relevant points all lie in \(\Delta_q\), and hence
\[
\frac1{|\Lambda_Q|}
\sum_{a\in\Lambda_Q}f_h^{L+1}(x+a)
=
\frac1{|\Lambda_Q|}
\sum_{a\in\Lambda_Q}\widetilde f_h^L(x+a)
\geq
\widetilde X_h^L(x)
=
X_h^{L+1}(x).
\]
Similarly,
\[
f_h^{L+1}(x+y-Q)
=
\widetilde f_h^L(x+y-Q)
=
\widetilde f_h^L(x+y)
=
f_h^{L+1}(x+y).
\]

Now let \(h=K\). Since \(\Delta_q\cap\Psi_q=\varnothing\),
\[
X_K^{L+1}(x)
=
(1-\gamma)^{-1}\widetilde X_K^L(x)1_{\Gamma_s}(x)
\leq
(1-\gamma)^{-1}\widetilde X_K^L(x).
\]
Consequently,
\[
\begin{aligned}
\frac1{|\Lambda_Q|}
\sum_{a\in\Lambda_Q}f_K^{L+1}(x+a)
&=
\frac{1}{1-\gamma}
\frac1{|\Lambda_Q|}
\sum_{a\in\Lambda_Q}\widetilde f_K^L(x+a)\\
&\geq
\frac{1}{1-\gamma}\widetilde X_K^L(x)
\geq X_K^{L+1}(x).
\end{aligned}
\]
The same factor occurs at both points in the periodicity identity, so
\[
\begin{aligned}
f_K^{L+1}(x+y-Q)
&=(1-\gamma)^{-1}\widetilde f_K^L(x+y-Q)\\
&=(1-\gamma)^{-1}\widetilde f_K^L(x+y)
=f_K^{L+1}(x+y).
\end{aligned}
\]

\medskip

Suppose next that \(x\notin\Delta_q\). Then $Q=\bar Q'_x$ and  $x\in(-\Lambda_q)^\gamma$,
the latter following from
\(\mathbb Z_q\setminus(-\Lambda_q)^\gamma\subseteq\Delta_q\).

Let \(1\leq h\leq K-1\). Then
\(x\in(-\Lambda_q)^\gamma\setminus\bar E'\),
\(Q=Q_{x,L+1}=\bar Q'_x\), and
\(X_h^{L+1}(x)=\bar Z_h(x)\). Therefore, by
\cite[Lemma~7.1]{L9}, applied exactly as in
\cite[Section~8, verification of property~(4)]{L9},
\[
\frac1{|\Lambda_Q|}
\sum_{a\in\Lambda_Q}f_h^{L+1}(x+a)
\geq
\bar Z_h(x)
=
X_h^{L+1}(x).
\]
For local periodicity, note that \(q\mid Q\). Hence \(x+y-Q\) and \(x+y\)
have the same residue modulo \(q\), so they either both belong to \(\Delta_q\)
or both lie outside \(\Delta_q\). In the latter case the required identity
follows from the restricted family:
\[
\bar g_h(x+y-Q)=\bar g_h(x+y).
\]
In the former case both values are given by \(\widetilde f_h^L\). Since
\(\widetilde f_h^L\) is \(pT_L\)-periodic and \(pT_L\mid Q\), we again obtain
\[
f_h^{L+1}(x+y-Q)=f_h^{L+1}(x+y).
\]

Finally, let \(x\notin\Delta_q\) and \(h=K\). If \(x\in\Psi_q\), then,
since \(pqT_L\mid Q\), \cite[Lemma~6.1]{L9} gives
\[
\frac1{|\Lambda_Q|}
\sum_{a\in\Lambda_Q}f_K^{L+1}(x+a)
\geq\alpha
\geq\alpha1_{\Gamma_t}(x)
=X_K^{L+1}(x).
\]
If \(x\notin\Psi_q\), then \(X_K^{L+1}(x)=0\), and residue domination
follows immediately from \(f_K^{L+1}\geq0\).

Moreover, \(f_K^{L+1}\) is periodic with period \(pqT_L\), because
\(\widetilde f_K^L\) is \(pT_L\)-periodic and \(\Delta_q\) is \(q\)-periodic.
Since
\[
pqT_L\mid\bar Q'_x=Q,
\]
we have, for every \(y\),
\[
f_K^{L+1}(x+y-Q)=f_K^{L+1}(x+y).
\]

This proves residue domination and local periodicity in all four cases.

	Therefore,  all four properties of
	\(\operatorname{Step}_{\Psi}(K,M,L+1)\) hold. This completes the induction, and
	therefore proves the \(\Psi\)-version of \cite[Proposition~4.1]{L9}. Taking
	\(L=M\) gives Proposition~\ref{p:lavictoire-short-control}.
\end{proof}

We now  extract
Proposition~\ref{p:finite-weak-failure} from the LaVictoire-type family. The
difficult construction has already been isolated in
Proposition~\ref{p:lavictoire-short-control}. What remains is to convert its
residue domination into a short-interval maximal lower bound and then amplify by
the pairwise independence of the \(X_h\)'s.
The only new conversion is the following. If \(x\notin E\), then the proposition
supplies a stopping modulus \(Q_x\). By the choice of \(N(Q_x)\) and by the
definition of \(\Psi\), for every \(1\leq m\leq H_{N(Q_x)}\),
\[
Q_x
\leq
\bigl(N(Q_x)+m\bigr)^d
\leq
\Psi(Q_x).
\]
Hence the local \(Q_x\)-periodicity allows us to reduce
\(f_h(x+(N(Q_x)+m)^d)\) modulo \(Q_x\). The short-block counting lemma then
converts the residue domination
\[
\frac1{|\Lambda_{Q_x}|}
\sum_{a\in\Lambda_{Q_x}}f_h(x+a)
\geq
X_h(x)
\]
into a lower bound for the short-interval average with left endpoint \(N(Q_x)\).

\begin{proof}[Proof of Proposition~\ref{p:finite-weak-failure} from
	Proposition~\ref{p:lavictoire-short-control}]
	Let \(C>0\) be given. We shall construct a cyclic example for which the
	weak \((1,1)\) ratio is larger than \(C\).
	Let \(B>1\) be a large parameter, to be chosen at the end in terms of \(C\).
	Choose \(0<\gamma<\gamma_0\) dyadic and sufficiently small, and set
	\[
	M:=\left\lfloor \frac{B}{\gamma}\right\rfloor .
	\]
	We also choose \(K\) as in LaVictoire's amplification step,
	\[
	K:=\left\lfloor \frac{\gamma}{B\varepsilon}\right\rfloor ,
	\]
	where \(\varepsilon=\varepsilon_\gamma\) is the parameter in
	Proposition~\ref{p:lavictoire-short-control}. Taking \(\gamma\) sufficiently
	small ensures that \(K\) is as large as needed. This is exactly the role of the
	small-\(\gamma\) condition in LaVictoire's proof.
	
	Apply Proposition~\ref{p:lavictoire-short-control} with \(L=M\), \(A=1\),
	\(D=1\), and with \(\delta\le 1/4\) sufficiently small. We obtain a cyclic group
	\(\mathbb Z_{\mathcal T}\), non-negative functions \(f_1,\dots,f_K\), random
	variables \(X_1,\dots,X_K\), an exceptional set \(E\), and a stopping modulus
	\(Q_x\). Define
	\[
	f:=\sum_{h=1}^K f_h .
	\]
	Since $\mathbb E f_h\leq (1+4\varepsilon)^{KM}$ and $KM
	\leq
	\frac{\gamma}{B\varepsilon}\cdot \frac{B}{\gamma}
	=
	\frac1{\varepsilon},$
	we have
	\[
	\mathbb E f_h
	\leq
	(1+4\varepsilon)^{1/\varepsilon}
	\leq e^4
	\]
	after decreasing \(\gamma\), if necessary. Hence
	\[
	\|f\|_{\ell^1(\mathbb Z_{\mathcal T})}
	=
	\mathbb E f
	\leq
	K e^4 .
	\]
	
We next convert the residue domination \eqref{e:rdomest} into domination by a
short-interval average. Fix \(x\notin E\), and put
\[
Q:=Q_x,
\qquad
N_x:=N(Q_x),
\qquad
H_x:=H_{N_x}.
\]
By the definition of \(N(Q)\), we have \(H_x\geq 4Q\). Also, by the definition
of the control function \(\Psi\), for every \(1\leq m\leq H_x\),
\begin{align}\label{e:sizeofQpsi}
    Q
\leq
(N_x+m)^d
\leq
(N_x+H_x)^d
\leq
\Psi(Q).
\end{align}
Therefore the local \(Q\)-periodicity from
Proposition~\ref{p:lavictoire-short-control} applies with $y=(N_x+m)^d .$

We apply Lemma~\ref{l:short-block-count} with
\[
U=N_x,
\qquad
H=H_x,
\qquad
Q=Q_x.
\]
Thus, for every \(a\in\Lambda_Q\),
\[
\#\{1\leq m\leq H_x:(N_x+m)^d\equiv a\pmod Q\}
\geq
\frac{H_x}{4|\Lambda_Q|}.
\]
For each residue class \(a\in\Lambda_Q\), let \(b(a)\in\{1,\dots,Q\}\)
denote its standard representative. If
\[
(N_x+m)^d\equiv a\pmod Q,
\]
then $(N_x+m)^d=b(a)+tQ$
for some integer \(t\geq0\). Since \(x\notin E\), Proposition~\ref{p:lavictoire-short-control}
gives the local \(Q\)-periodicity identity
\[
f_h(x+y-Q)=f_h(x+y),
\qquad
Q\leq y\leq \Psi(Q).
\]
By \eqref{e:sizeofQpsi}, all intermediate values \(b(a)+sQ\), \(1\leq s\leq t\),
lie in this range:
\[
Q\leq b(a)+sQ\leq b(a)+tQ=(N_x+m)^d\leq\Psi(Q).
\]
Hence repeated use of the local \(Q\)-periodicity gives
\[
f_h\bigl(x+(N_x+m)^d\bigr)
=
f_h(x+b(a)).
\]
Hence
\begin{align*}
\frac1{H_x}\sum_{m=1}^{H_x} f_h\bigl(x+(N_x+m)^d\bigr)
&\geq
\frac1{H_x}
\sum_{a\in\Lambda_Q}
f_h(x+b(a))
\#\{1\leq m\leq H_x:(N_x+m)^d\equiv a\pmod Q\} \\
&\geq
\frac1{4|\Lambda_Q|}
\sum_{a\in\Lambda_Q} f_h(x+b(a)) \geq
\frac14 X_h(x).
\end{align*}
	Summing over \(h\), we obtain
	\[
	\frac1{H_x}
	\sum_{m=1}^{H_x}
	f\bigl(x+(N_x+m)^d\bigr)
	\geq
	\frac14\sum_{h=1}^K X_h(x).
	\]
	Since \(N_x=N(Q_x)\in\mathcal N\), this gives the pointwise maximal lower bound
	\[
	\sup_{N\in\mathcal N}
	\frac1{H_N}
	\sum_{m=1}^{H_N}
	f\bigl(x+(N+m)^d\bigr)
	\geq
	\frac14\sum_{h=1}^K X_h(x)
	\]
	for every \(x\notin E\).
	
It remains to run LaVictoire's standard amplification argument
\cite[proof of Proposition~4.1]{L9}. We include the short calculation in order
to keep track of the constants. Since
\(X_h\stackrel{d}{=}Y_{M,\gamma,\alpha}\), we have
\[
\mathbb E X_h=1+M\alpha^2\gamma.
\]
By taking \(\gamma\) sufficiently small, we may assume that
\(M\gamma\geq B/2\), and hence
\[
\mathbb E X_h\geq \frac{B\alpha^2}{2}, \qquad \mathbb E X_h^2
\leq
2(1-\gamma)^{-M}
\leq
2e^{2B}.
\]
Since \(X_1,\dots,X_K\) are pairwise independent, Chebyshev's inequality gives
\[
\begin{aligned}
\mathbb P\left(
\frac1K\sum_{h=1}^K X_h
\leq
\frac{B\alpha^2}{4}
\right)
&\leq
\mathbb P\left(
\left|
\frac1K\sum_{h=1}^K X_h-\mathbb E X_1
\right|
\geq
\frac{B\alpha^2}{4}
\right)  \\
&\leq
\left(\frac{B\alpha^2}{4}\right)^{-2}
\frac1K\mathbb E (X_1-\mathbb E X_1)^2 \\
&\leq
\frac{32e^{2B}}{KB^2\alpha^4}.
\end{aligned}
\]
Choosing \(\gamma\) smaller if necessary makes \(K\) large enough that the last
quantity is at most \(1/2\). Therefore
\[
\mathbb P\left(
\sum_{h=1}^K X_h
\geq
\frac{KB\alpha^2}{4}
\right)
\geq
\frac12.
\]
Since \(\mathbb P(E)\leq\delta<1/4\), the set
\[
S:=
\left\{
x\notin E:
\sum_{h=1}^K X_h(x)
\geq
\frac{KB\alpha^2}{4}
\right\}
\]
has measure at least \(1/4\). On this set, the maximal lower bound already
proved gives
\[
\sup_{N\in\mathcal N}
\frac1{H_N}
\sum_{m=1}^{H_N}
f\bigl(x+(N+m)^d\bigr)
\geq
\frac{KB\alpha^2}{16}.
\]
Hence
\[
\left\|
\sup_{N\in\mathcal N}
\left|
\frac1{H_N}
\sum_{m=1}^{H_N}
f\bigl(x+(N+m)^d\bigr)
\right|
\right\|_{\ell^{1,\infty}(\mathbb Z_{\mathcal T})}
\geq
\frac{KB\alpha^2}{64}.
\]
Since \(\|f\|_{\ell^1(\mathbb Z_{\mathcal T})}\leq Ke^4\), we obtain
\[
\left\|
\sup_{N\in\mathcal N}
\left|
\frac1{H_N}
\sum_{m=1}^{H_N}
f\bigl(x+(N+m)^d\bigr)
\right|
\right\|_{\ell^{1,\infty}(\mathbb Z_{\mathcal T})}
\geq
\frac{B\alpha^2}{64e^4}
\|f\|_{\ell^1(\mathbb Z_{\mathcal T})}.
\]
Finally choose \(B\) so large that
\[
\frac{B\alpha^2}{64e^4}>C.
\]
This proves Proposition~\ref{p:finite-weak-failure}.
\end{proof}

	\section{Failure of pointwise convergence along dense times}

	%\section{counterexample}  
	%By \cite[Corollary 5]{BJR_90_moving_avgs}, for every $p \geq 1$ and for $0< c< 1$, 
	%\begin{align*}
	%	\lim_{k \rightarrow \infty} \frac{1}{2^{kc}} \sum_{j=0}^{ \lfloor 2^{kc} - 1 \rfloor } f(T^{2^k + j})
	%\end{align*}
	%does not exist a.e. for some $f \in L^p$. The above averages satisfy the so-called `strong sweeping out property' (which is a stronger property than divergence for $f \in L^p$ for some $p \geq 1$). 
	%
	%
	%Instead of extending the construction in \cite{BJR_90_moving_avgs}, let us instead proceed in the following way. 
	
	In this section we prove that for an aperiodic measure-preserving transformation $T$ on a non-atomic Lebesgue probability space $(X, \mathcal{B}, \mu)$, the averages 
	\begin{align*}
		A_{j_N}^{} f(x) = \frac{1}{\floor{j_N^c}} \sum_{n\in [j_N, j_N + j_N^c)} f(T^{P(n)}x) = \frac{1}{\floor{j_N^c}} \sum_{ n \in  [0,  j_N^c)} f(T^{P(j_N + n)}x)
	\end{align*}
	do not converge a.e. for some $f \in L^p$, where $p \geq 1$. In fact we will show that the averages $A_{j_N}$ have the \textit{strong sweeping out property} provided that $j_N$ satisfy certain growth conditions. The main technical tool we use is \cite[Theorem 1.14]{SSOP_96}. To this end,  for the convenience of the reader we reiterate some relevant definitions.
	
	\begin{definition}
		A sequence of operators $(T_n)$ has the \textit{strong sweeping out property} if given $\ep>0$ there is a set $B$ with $\mu(B) < \ep$ such that
		\begin{align*}
			\limsup_{n \rightarrow \infty} T_n \chi_B (x) =1 \ \ \textrm{a.e.}  \hspace{1in} \liminf_{n \rightarrow \infty} T_n \chi_B (x) =0 \ \ \textrm{a.e.} 
		\end{align*}
	\end{definition}

	\begin{definition}
		A sequence of measures $\left(\nu_n\right)_{n \in \N}$ supported on a countable set $\mathcal{W}$ is called \textit{dissipative} if for each $w \in \mathcal{W}$ we have that \[ \lim_{N \to \infty} \nu_N(w)=0.  \]
	\end{definition} 
	\begin{proposition} \label{p:sweepingout}
		Let $\left\{j_n\right\}_{n \in \N}$ be a subsequence of the positive integers with the property that \[ \lim_{n \to \infty} \frac{\log\left(j_{n+1}\right)}{\log\left(j_n\right)}=1, \quad \lim_{n \to \infty} j_n=+ \infty.  \] Then the averages $\left\{A_{j_n}^{P,c}\right\} $ have the strong sweeping out property for all $P \in \Z[x]$ with $\deg(P) \geq 2$ and all $c \in (0,1).$
	\end{proposition}
	\begin{proof}
		Initially, let us note that the associated probability measures and their corresponding Fourier transforms are 
		\begin{align*}
			K_{j_N} (m) = \frac{1}{\floor{j_N^c}} \sum_{n \in [1,  j_N^c]} \de_{P(j_N + n)} (m), \hspace{.2in} \hspace{.2in} m_{j_N}(\xi) = \frac{1}{\floor{j_N^c}} \sum_{n \in [1, j_N^c]} e(P(j_N+n) \xi ) . 
		\end{align*} 	It remains to check that the hypotheses of \cite[Theorem 1.14]{SSOP_96} apply. For $\left( K_{j_N}\right)_{N \in \N}$ we have that $\mathcal{W}=\Z$.  Fixing $m \in \Z$ we have that for large $N$, $K_{j_N}(m)=0$ so the dissipative property is immediate. It thus remains to find $\delta_0>0$ s.t. given $N\in \N$ large, and a small number $1 >  \eta>0$, we can find $\alpha, \beta$ s.t. $\delta_0< |\alpha - \beta| < 1 - \delta_0$, integers $k_1< \ldots< k_N$, $j_{k_1}< \ldots < j_{k_N}$, and frequencies $\xi_1, \ldots, \xi_N$ such that 
		\begin{align*}
			\begin{cases}
				|m_{j_{k_i}} (\xi_i) - e(\alpha)| \lesssim \eta \hspace{.2in} i=1, \ldots, N\\
				|m_{j_{k_m}}(\xi_i) - e(\beta)|  \lesssim \eta \hspace{.2in} i \neq m. 
			\end{cases}
		\end{align*}
		Let $\alpha =0$, $\beta=\frac{a_d}{2^\zeta}$ where $\zeta$ is chosen such that  $  \frac{1}{4}<\frac{|a_d|}{2^\zeta} \leq \frac{1}{2}. $  The choice of $\delta_0$ is   $\delta_0 = \frac{1}{4}$.

		Observe that 
		\begin{equation} \label{eq:phase}
			m_{j_k}(\xi) = \frac{1}{\floor{j_k^c}} \sum_{n \in [1,  j_k^c]} e(P(j_k+n)\xi ) = e(a_d j_k^d\xi )+ O \left( j_k^{d-1+c}\xi  \right)
		\end{equation}
		since one can easily note that 	$P(j_k+n)\xi=a_dj_k^d\xi+O\left(j_k^{d-1+c}\xi\right)$. 
		
		We choose parameters $\tau_1,\tau_2>1$ very close to $1$ with the properties that \[\exists n_0: \;\forall n \geq n_0 \quad j_{n+1} \leq  j_{n}^{\tau_1}, \quad (\tau_1 \tau_2)^N< \frac{d}{d-1+c} \] and an index $k_1 \geq n_0$ very large so that $\max \left\{ j_{k_1}^{1-\tau_2},j_{k_1}^{(\tau_1\tau_2)^N(d-1+c)-d}\right\} \leq \frac{\eta}{5N|a_d|} .$  We define the sequence $\left\{k_i\right\}_{i \in [N]}$ recursively by the formula \[ k_{i+1}= \min \left\{ n \geq k_i: j_{n}\geq j_{k_i}^{\tau_2 } \right\}. \]  The construction allows us to conclude that $j_{k_i}^{\tau_2} \leq j_{k_{i+1}} \leq j_{k_i}^{\tau_1\tau_2} $	which by iterating yields \[j_{k_i} \geq j_{k_{i-t}}^{\tau_2^{t}}, \quad j_{k_N} \leq j_{k_1}^{(\tau_1\tau_2)^{N-1}} \]
		
		We define the frequencies \[ \xi_i=\sum_{ a \in [N]}  \frac{A_{i}^a}{j_{k_a}^{d}} \] where $\left\{A_{i}^a\right\}_{a,i \in [N]}$ are elements of the torus and  are defined recursively as \[ A_{i}^a= \left\{\begin{matrix}  \left\{\frac{1}{2^\zeta}-\sum_{t<a} A_{i}^t \left(\frac{j_{k_a}}{j_{k_t}}\right)^d\right\}, \quad a \neq i
			\\  \left\{-\sum_{t<a} A_{i}^t \left(\frac{j_{k_a}}{j_{k_t}}\right)^d\right\}, \quad a =i. 
		\end{matrix}\right.    \] 
		
		The construction above enables us to verify the following crucial  property  \begin{equation} \label{eq:phase2}
			\left| m_{j_{k_i}}(\xi_{m})-e(a_dj_{k_i}^d\xi_{m}) \right| \leq  \frac{\eta}{5}, \forall m,i \in [N]
		\end{equation}
		
		which follows directly  from $j_{k_i}^{d-1+c}|\xi_m| \leq  \frac{\eta}{5}, \forall m,i \in [N] $,   a byproduct of the fact that $ j_{k_i}^{d-1+c}|\xi_m| \leq   N \frac{j_{k_1}^{(\tau_1\tau_2)^{N-1}(d-1+c)}}{j_{k_1}^d} \leq \frac{\eta}{5}$,  and \eqref{eq:phase}. Finally, it is clear that the proof will be complete once we verify the estimates below simply by combining them with \eqref{eq:phase2}   \[ \begin{split}
			&	\left \| j_{k_i}^d \xi_i \right \|_{\TT} \leq \frac{\eta }{5}, \quad  \left \| a_{d} \left( j_{k_i}^d \xi_m-\frac{1}{2^{\zeta}} \right) \right \|_{\TT} \leq \frac{\eta}{5}, \; m \neq i.
		\end{split}  \] Utilizing the construction of the coefficients $A_{i}^a$ and therefore the frequencies as well as the growth of $\left\{j_{k_i}\right\}_{i \in [N]}$ we have that \[ \begin{split}
			&\left \| j_{k_i}^d \xi_i \right \|_{\TT}= \left \| \sum_{t > i } A_{i}^t \left(\frac{j_{k_i}}{j_{k_t}}\right)^d \right \|_{\TT}  \leq \sum_{t>i} \left(\frac{j_{k_i}}{j_{k_i}^{\tau_2^{t-i}}}\right)^d \leq \frac{N}{j_{k_i}^{\tau_2-1}} \leq \frac{\eta}{5} \\ & \left\| a_d \left( j_{k_i}^d \xi_m-\frac{1}{2^{\zeta}} \right) \right\|_{\TT}= \left \| a_d \sum_{t>i} A_{m}^t\left(\frac{j_{k_i}}{j_{k_t}}\right)^d   \right \|_{\TT} \leq \frac{\eta}{5}.
		\end{split} \]
	\end{proof}
	
	\bibliography{references}
	\bibliographystyle{amsplain}
\end{document}